\documentclass[11pt]{amsart}
\usepackage{geometry}
\usepackage{amssymb, amsthm, amsfonts, mathrsfs}
\usepackage{blindtext}
\usepackage{comment}

\DeclareSymbolFont{euletters}{U}{eur}{m}{n}
\DeclareSymbolFont{eufrakletters}{U}{euf}{m}{n}
\DeclareFontFamily{U}{wncy}{}
    \DeclareFontShape{U}{wncy}{m}{n}{<->wncyr10}{}
    \DeclareSymbolFont{mcy}{U}{wncy}{m}{n}
    \DeclareMathSymbol{\Sha}{\mathord}{mcy}{"58}

\usepackage{tabularx}
\usepackage{latexsym}
\usepackage{color}
\usepackage[all,cmtip,line]{xy}

\usepackage[only, llbracket,rrbracket]{stmaryrd}
\usepackage{todonotes}

\usepackage[shortlabels]{enumitem} 
\usepackage{xcolor}
\usepackage{bbm}
\usepackage{pdfpages}
\usepackage{titletoc}
\usepackage{booktabs}
\usepackage{mathtools}
\usepackage{thmtools}
\usepackage{hhline}
\usepackage{caption}
\usepackage{multirow, url}
\usepackage{textcomp}
\DeclareMathAlphabet{\cmcal}{OMS}{cmsy}{m}{n}

\usepackage{arydshln}

\usepackage{libertine}
\usepackage[T1]{fontenc}
\usepackage[utf8]{inputenc}

\definecolor{blue1}{rgb}{0, 0, 2}
\definecolor{sky}{rgb}{0, 0.2, 0.8}
\usepackage[colorlinks=true, citecolor=sky, linkcolor=magenta, urlcolor=blue1, filecolor=cyan, backref=page]{hyperref}

\newtheorem{theorem}{Theorem}[section]

\newtheorem{lemma}[theorem]{Lemma}
\newtheorem{proposition}[theorem]{Proposition}

\theoremstyle{definition}
\newtheorem{definition}[theorem]{Definition}

\theoremstyle{remark}
\newtheorem{remark}[theorem]{\bf{Remark}}

\numberwithin{equation}{section} \numberwithin{table}{subsection}
	
\newtheorem*{theorem*}{\bf{Theorem}}
\newtheorem*{claim*}{\bf{Claim}}
\newtheorem*{convention*}{\bf{Convention}}
\newtheorem*{remark*}{\bf{Remark}}
\newtheorem*{remarks*}{\bf{Remarks}}
\newtheorem*{example*}{\bf{Example}}
\newtheorem*{examples*}{\bf{Examples}}

\newcommand{\C}{{\mathbb{C}}}

\newcommand{\F}{{\mathbb{F}}}

\newcommand{\Q}{{\mathbb{Q}}}
\newcommand{\R}{{\mathbb{R}}}

\newcommand{\Z}{{\mathbb{Z}}}

\newcommand{\cF}{{\cmcal{F}}}

\newcommand{\cO}{{\cmcal{O}}}
\newcommand{\cP}{{\mathcal{P}}}

\newcommand{\cS}{{\cmcal{S}}}

\def\a{\alpha}
\def\b{\beta}

\def\d{\delta}

\def\ve{\varepsilon}
\def\g{\gamma}

\def\D{\Delta}

\def\<{\left\langle}
\def\>{\right\rangle}

\newcommand{\zmod}[1]{{\Z/{#1}\Z}}

\newcommand{\inj}{\hookrightarrow}
\newcommand{\surj}{\twoheadrightarrow}
\newcommand{\arinj}{\ar@{^(->}}
\newcommand{\arsurj}{\ar@{->>}}
\newcommand{\arsub}{\ar@{}[r]|-*[@]{\subset}}
\newcommand{\arsup}{\ar@{}[r]|-*[@]{\supset}}
\newcommand{\arcap}{\ar@{}[d]|-*[@]{\subset}}
\newcommand{\arcup}{\ar@{}[u]|-*[@]{\subset}}
\newcommand{\arin}{\ar@{}[u]|-*[@]{\in}}

\renewcommand{\pmod}[1]{{\,(\textnormal{mod}\hspace{1mm} {#1})}}

\newcommand{\Hom}{{\textnormal{Hom}}}
\newcommand{\Gal}{{\textnormal{Gal}}}

\newcommand{\sHom}{{\mathscr{H}\kern-.5pt om}}
\newcommand{\sExt}{{\mathscr{E}\kern-.5pt xt}}

\newcommand{\sgn}{\textnormal{sgn}}

\newcommand{\ns}{{N=\square}}

\renewcommand{~}{\hspace*{0.5mm}}

\newcommand{\wt}{\widetilde}

\newcommand{\ov}{\overline}

\newcommand{\ms}{\medskip}
\newcommand{\sm}{\smallsetminus}

\renewcommand{\dim}{\textnormal{dim}~}

\newcommand{\Sel}{\textnormal{Sel}}
\newcommand{\sign}{\textnormal{sgn}}
\newcommand{\im}{\textnormal{im}}
\renewcommand{\ker}{\textnormal{ker}}
\newcommand{\sel}[1]{{#1^\times}/{({#1^\times})^2}}

\mathchardef\hyp="2D
\newcommand{\xyv}[1]{\xymatrixrowsep{#1 pc}}
\newcommand{\xyh}[1]{\xymatrixcolsep{#1 pc}}

\newcommand{\qa}{{\quad \text{and} \quad}}
\newcommand{\qqa}{{~ \text{ and } ~}}

\newcommand{\magenta}[1]{{\color{magenta} #1}}

\renewcommand{\text}{\textnormal}

\begin{document}
\def\UrlFont{\rm}
\title{Bounds for $2$-Selmer ranks in terms of seminarrow class groups}

\author{Hwajong Yoo}

\address{Hwajong Yoo, College of Liberal Studies and Research Institute of Mathematics, Seoul National University, Seoul 08826, South Korea}
\email{\textnormal{hwajong@snu.ac.kr}}                                                                  

\author{Myungjun Yu}
\address{Myungjun Yu, Center for Mathematical Challenges, Korea Institute for Advanced Study, 85 Hoegi-ro, Dongdaemun-gu, Seoul, South Korea}
\email{\textnormal{mjyu.math@gmail.com}}                                                                  
\maketitle
\begin{abstract}
Let $E$ be an elliptic curve over a number field $K$ defined by a monic irreducible cubic polynomial $F(x)$. When $E$ is \textit{nice} at all finite primes of $K$, we bound its $2$-Selmer rank in terms of the $2$-rank of a modified ideal class group of the field $L=K[x]/{(F(x))}$, which we call the \textit{semi-narrow class group} of $L$. We then provide several sufficient conditions for $E$ being nice at a finite prime. 

As an application, when $K$ is a real quadratic field, $E/K$ is semistable and the discriminant of $F$ is totally negative, then we frequently determine the $2$-Selmer rank of $E$ by computing the root number of $E$ and the $2$-rank of the narrow class group of $L$.  
\end{abstract}


\section{Introduction}\label{section: introduction}
Let $E$ be an elliptic curve over a number field $K$, given in the form $y^2=F(x)$ where $F(x)$ is a monic cubic polynomial with coefficients in $\cO_K$, the ring of integers of $K$.
The Mordell--Weil theorem tells us that the $K$-rational points $E(K)$ form a finitely generated abelian group. 
The rank of $E(K)$, called the \textit{Mordell--Weil rank}, is one of the central objects in number theory.
Unfortunately, there is no known general algorithm that is guaranteed to find the Mordell--Weil rank. One of the most common methods for computing it is studying the $2$-\textit{Selmer group} of $E$, denoted by $\Sel_2(E/K)$, which is effectively computable.

From now on, we assume that $|E(K)[2]|=1$, i.e., $F(x)$ is irreducible over $K$. Let $L:=K[x]/{(F(x))}$ be a cubic extension of $K$. It is known that there should be a connection between the $2$-Selmer group of $E$ and the $2$-class group of $L$. For a description of known results, see the introduction of \cite{BPT}. 
Our main goal of this article is to understand this connection more thoroughly. To do so, we first identify\footnote{This is well-known, for example, Case 1 of \cite[p. 717]{BK77}. For details, see \cite[p. 9]{St17} or \cite[Lem. 2.7]{Li19}.} $H^1(K, E[2])$ with
\begin{equation*}
(\sel L)_\ns := \{ [\a] \in \sel L : N(\a) \in (K^\times)^2 \},
\end{equation*}
where $N:L^\times \to K^\times$ is the norm map. Similarly, we identify $H^1(K_v, E[2])$ with $(\sel {L_v})_\ns$, where  
\begin{equation*}
L_v:=L \otimes_K K_v = K_v[x]/{(F(x))}.
\end{equation*}
Then we can regard the $2$-Selmer group as a subgroup of $(\sel L)_\ns$, i.e., we define the $2$-Selmer group of $E$ as follows:
\begin{equation*}
\Sel_2(E/K) := \{ [\a] \in (\sel L)_\ns : [\a_v] \in \im (\d_{K_v}) \text{ for all primes $v$ of $K$}\},
\end{equation*}
where $\d_{K_v} : E(K_v)/{2E(K_v)} \inj H^1(K_v, E[2])=(\sel {L_v})_\ns$ is the \textit{local Kummer map}. (For unfamiliar notation, see Section \ref{subsection: notation}.) From now on, we call $\im(\d_{K_v})$ the \textit{local condition for $\Sel_2(E/K)$}.

\ms
Now, we consider subgroups of $(\sel L)_\ns$ which are related to $C_L$, the ideal class group of $L$.
Following \cite[Lem. 2.16]{Li19} we may define
\begin{equation*}
M_1':=\{ [\a] \in (\sel L)_\ns : L(\sqrt{\a})/L \text{ is unramified everywhere}\}
\end{equation*}
and
\begin{equation*}
M_2':=\{ [\a] \in (\sel L)_\ns : (\a)=I^2 \text{ for some } I \in \cF_L \qqa  \a \gg 0 \},
\end{equation*}
where $\cF_L$ is the group of fractional ideals of $L$. When $K=\Q$, we have the following \cite[Th. 2.18]{Li19}. 
\begin{theorem}[Li]
\label{theorem: li}
Suppose that $K=\Q$ and the discriminant of $F$ is negative and squarefree. Then we have
\begin{equation*}
M_1' \subset \Sel_2(E/\Q) \subset M_2', ~~|M_1'|=|C_L[2]| \qa [M_2' : M_1']=2.
\end{equation*}
Thus, we have
\begin{equation*}
\dim_{\F_2} C_L[2] \leq \dim_{\F_2} \Sel_2(E/\Q) \leq \dim_{\F_2} C_L[2]+1.
\end{equation*}
\end{theorem}
This theorem says that if we know $\dim_{\F_2} C_L[2]$ then $\dim_{\F_2} \Sel_2(E/\Q)$, which is called the {\it $2$-Selmer rank of $E$}, is completely determined by its \emph{root number}. As in Theorem \ref{theorem: li}, we wish to have $M_1' \subset \Sel_2(E/K) \subset M_2'$ for other number fields $K$ or other polynomials $F$ with more relaxed hypothesis. However, it cannot be achieved in general if there is a real prime $v$ of $K$ that is unramified in $L$. So we instead allow the ramifications at some real primes above unramified real primes of $K$ and 
consider new subgroups of $(\sel {L})_\ns$, which are related to a modified ideal class group of $L$.
\begin{definition}\label{definition: semi-narrow class group}
Let $P_L^\infty$ be the group of elements in $L^\times$ satisfying some positivity conditions, which is defined in Section \ref{subsection : semi-narrow class group}. We define the \textit{semi-narrow class group of $L$} by
\begin{equation*}
C^\infty_L:=\cF_L/{\{(\a) : \a \in P_L^\infty\}}.
\end{equation*}
Also, let
\begin{equation*}
\begin{split}
M_1:=\{ [\a] \in (\sel L)_\ns : ~\text{$L(\sqrt{\a})/L$ is unramified at all finite primes and } \a \in P_L^\infty \}
\end{split}
\end{equation*}
and
\begin{equation*}
M_2:=\{ [\a] \in (\sel L)_\ns : ~(\a)=I^2 \text{ for some } I \in \cF_L \text { and } \a \in P_L^\infty \}.
\end{equation*}
\end{definition}
Then we have the following \cite[Th. 2.16]{BPT}.
\begin{theorem}[Barrera--Pacetti--Tornar\'ia]
\label{theorem: BPT}
Suppose that the narrow class number of $K$ is odd, and $E/K_v$ satisfies certain conditions for all finite primes $v$ of $K$.  
Then we have
\begin{equation*}
M_1 \subset \Sel_2(E/K) \subset M_2, ~~ |M_1| = |C^\infty_L[2]| \qa [M_2:M_1] \leq 2^{[K:\Q]}.
\end{equation*}
Thus, we have
\begin{equation*}
\dim_{\F_2} C^\infty_L[2] \leq \dim_{\F_2}\Sel_2(E/K) \leq \dim_{\F_2}  C^\infty_L[2] + [K:\Q].
\end{equation*}
\end{theorem}
Their result indeed covers a lot larger class of elliptic curves $E/K$ than the previous work \cite{BK77, Li19}.
In spite of that, the assumption that the narrow class number of $K$ is odd is somewhat restrictive. 
For example, it is known that at least $50\%$ of totally real cubic fields have even narrow class number \cite[Cor. 7]{BV15}. For real quadratic fields, even worse is true: $100\%$ of them have even narrow class number \cite[Th. 5]{BV15}. Therefore one may hope to remove this hypothesis. 

In the present article, we generalize Theorem \ref{theorem: BPT} to the case when $K$ is an arbitrary number field. First, we compute the sizes of $M_1$ and $M_2$ for any number field $K$ in terms of the semi-narrow class group of $L$.
\begin{theorem}\label{theorem: main theorem sizes of M1 and M2}
We have
\begin{equation*}
|M_1|=\frac{|C^\infty_L[2]|}{|C^+_K[2]|} \qa |M_2|=\frac{|C^\infty_L[2]|\times 2^{[K:\Q]}}{|C^+_K[2]|},
\end{equation*}
where $C^+_K$ is the narrow class group of $K$.
\end{theorem}

Next, we wish to understand when we have
\begin{equation*}
M_1 \subset \Sel_2(E/K) \subset M_2,
\end{equation*}
which provides bounds for the $2$-Selmer rank of $E$ by Theorem \ref{theorem: main theorem sizes of M1 and M2}.
For a finite prime $v$ of $K$, we first define $M_{i, v}$ as follows: Let
\begin{equation*}
M_{1, v}:=\{[\a] \in \left(\sel {L_v}\right)_\ns : L_v(\sqrt{\alpha})/{L_v} \text{ is unramified}\}
\end{equation*}
and
\begin{equation*}
M_{2, v}:=\{[\a] \in \left(\sel {L_v}\right)_\ns : \forall w \mid v, w(\alpha) \in 2\Z\},
\end{equation*}
where $w$ is a prime of $L$. (For the definition of $M_{i, v}$ for an infinite prime $v$ of $K$, see Section \ref{section: 3.1}.) 
Then we define
\begin{equation*}
M_i :=\{ [\a] \in (\sel L)_\ns : [\a_v] \in M_{i, v} \text{ for all primes $v$ of $K$}\}.
\end{equation*}
Note that if $v$ is an odd prime then $M_{1, v}=M_{2, v} = (\sel {\cO_{L_v}})_\ns$, where $\cO_{L_v}^\times$ denotes the unit group of the ring of integers $\cO_{L_v}$ of $L_v$.
Note also that if $v$ is an infinite prime then $M_{i, v}$ is defined so that $M_{1, v}=M_{2, v}=\im (\d_{K_v})$. 
\begin{definition}
\label{definition : loweruppernicenice}
For a finite prime $v$ of $K$, we say that an elliptic curve $E/K_v$ is \textit{lower nice} (resp. \textit{upper nice}) if $M_{1, v} \subset \im (\d_{K_v})$ (resp. $\im(\d_{K_v}) \subset M_{2, v}$).
If $E/K_v$ is both lower nice and upper nice, then we say that $E/K_v$ is \textit{nice}. 
Also, we say that an elliptic curve $E$ over a number field $K$ is \textit{lower nice at $v$} (resp. \textit{upper nice at $v$} and \textit{nice at $v$}) if $E/K_v$ is so. 
\end{definition}
Since the Selmer group is defined by the local conditions, we obtain the following.
\begin{theorem}\label{theorem: main theorem nice implies bounds}
If $E$ is lower nice at all finite primes of $K$, then we have $\dim_{\F_2} \Sel_2(E/K) \geq n$, where 
\begin{equation*}
n=\dim_{\F_2} C^\infty_L[2]-\dim_{\F_2} C^+_K[2].
\end{equation*}
Also, if $E$ is upper nice at all finite primes of $K$, then we have $\dim_{\F_2} \Sel_2 (E/K) \leq  n+[K:\Q]$.
Thus, if $E$ is nice at all finite primes of $K$, then we have
\begin{equation*}
n \leq \dim_{\F_2} \Sel_2(E/K) \leq n +[K:\Q].
\end{equation*}
\end{theorem}

\begin{remark}
As in \cite[Def. 3.1]{St17}, we may define 
\begin{equation*}
L(S, 2):=\{ [\a] \in \sel {L} : \forall v \not\in S,  \forall w \mid v : w(\a) \in 2\Z \},
\end{equation*}
where $S$ is the set of ``bad'' primes of $K$. Here, by ``bad'' primes we mean either the real infinite primes, even primes, or the primes of bad reduction for $E$. Then we have 
\begin{equation*}
\Sel_2(E/K) \simeq \{ [\a] \in L(S, 2) : N(\a) \in (K^\times)^2, \forall v \in S : [\a_v] \in \im (\d_{K_v}) \}.
\end{equation*}
It is easy to see that $M_2 \subset L(S, 2)$ and $L(S, 2)$ is much larger than $M_2$ in general.

In some sense, the groups $M_1$ and $M_2$ give the ``best possible bounds'' for the $2$-Selmer ranks of nice elliptic curves. 
As mentioned right before Definition \ref{definition : loweruppernicenice}, if $v$ is not even (including all the other ``bad'' primes) then the local conditions $M_{1,v}$ and $M_{2,v}$ coincide. 
Therefore the even primes are exactly the places where $M_1$ and $M_2$ differ.  
In general, however, it is extremely difficult to exactly compute $\im (\d_{K_v})$, the local condition of $\Sel_2(E/K)$ at $v$, for an even prime $v$. 
For such $v$, what one can do in some fortunate situations (which justifies the word ``nice'') is proving $\im(\d_{K_v})$ is a subset (resp. superset) of $M_{2,v}$ (resp. $M_{1,v}$). 
\end{remark}

Next, we discuss sufficient conditions for $E$ being nice. There are some cases dependent only on the field extension $L/K$.
\begin{proposition}[Barrera--Pacetti--Tornar\'ia]\label{proposition: BPT}
Let $v$ be a finite prime of $K$. Suppose that either $L_v$ is a cubic extension of $K_v$ or $\cO_{L_v} = \cO_{K_v}[x]/{(F(x))}$. Then $E$ is nice at $v$.
\end{proposition}

One case satisfying the latter condition is the following. (In general, it is not easy to check when the conditions in Proposition \ref{proposition: BPT} are satisfied.)
\begin{proposition}[Proposition \ref{proposition: valuation is at most 1}]
Let $D$ be the discriminant of $F$. If $v(D) \leq 1$, then $E$ is nice at $v$.
\end{proposition}

If we require additional hypothesis on $E/{K_v}$ we have the following \cite{BK77}.
\begin{theorem}[Brumer--Kramer]
For an odd prime $v$, $E$ is nice at $v$ if $[E(K_v) : E_0(K_v)]$ is odd. For an even prime $v$, $E$ is nice at $v$ if $K_v/{\Q_2}$ is unramified and $E$ has good reduction at $v$.
\end{theorem}

One of our main theorems is the following, which removes the condition on $K_v$.
\begin{theorem}[Theorems \ref{theorem: good reduction} and \ref{theorem : multiplicative reduction}]\label{thm}
For an even prime $v$, $E$ is nice at $v$ if one of the following holds.
\begin{enumerate}
\item 
$E$ has good ordinary reduction at $v$.
\item
$E$ has good supersingular reduction at $v$, $v(2)$ is not divisible by $3$, and 
either $v(a_1)$ is odd or $3v(a_1) \geq 2v(2)$, where $a_1$ is the coefficient of $xy$ in a Weierstrass minimal model of $E/{K_v}$.
\item
$E$ has multiplicative reduction at $v$ and $v(D)$ is odd.
\end{enumerate}
\end{theorem}
Note that in the case (2) we prove that $L_v$ is a cubic extension of $K_v$, so it  is a special case of Proposition \ref{proposition: BPT}. 
It remains an interesting question how sharp the conditions in Theorem \ref{thm} are, in particular, to find examples of $E$ which are not nice at $v$ when the additional requirement in (2) or (3) is violated.

\ms
As an application, we consider the following situation: Suppose that $K$ is quadratic. Then the conditions in (2) are automatically satisfied when $E$ has good supersingular reduction at even primes. Thus, if $E/K$ has semistable reduction at all even primes and the minimal discriminant of $E/K_v$ has odd or zero valuation for all primes $v$, then we may replace $L(S, 2)$ by $M_2$ in the computation of the $2$-Selmer rank of $E$. Furthermore, if the minimal discriminant of $E/K$ is totally negative then the semi-narrow class group of $L$ is equal to the narrow class group of $L$. Note that in SAGE \cite{Sa20} the computation of the narrow class group of $L$ is much faster than that of the $2$-Selmer rank of $E$.
In Section \ref{section: examples} we provide some examples in this direction.


\ms
\subsection{Notation}\label{subsection: notation} For an abelian group $A$ and its element $a$, let $[a]$ denote the coset represented by $a$ of the factor group $A/2A$ (or $A/{A^2}$ if the group law is written multiplicatively). 

Let $K$ be any number field. For a finite prime $v$ of $K$, we denote by $v:K_v^\times \to \Z$ the normalized valuation sending a uniformizer of $\cO_{K_v}$ to $1$. We often abuse the notation and write $v(\a)$ for $\a \in K^\times$ for the normalized valuation of the image of $\a$ in $K_v^\times$.  Also, we write $\a_v$ for the image of $\a$ by the completion $K \inj K_v$ when $v$ is a finite prime.
On the other hand, for an infinite prime $v$ of $K$ we denote by $v(\a)$ the image of $\a$ by the completion $K \inj K_v$.

We say a finite prime $v$ is \textit{even} (resp. \textit{odd}) if it lies above $2$ (resp. otherwise).

\ms
\section{Modified ideal class groups}\label{section: modified ideal class groups}
In this section, we introduce various modified class groups and compute the sizes of $M_1$ and $M_2$ in terms of a \textit{semi-narrow class group}.

\ms
As in the previous section, let $K$ be a number field and $L=K[x]/{(F(x))}$ a cubic extension of $K$.
\subsection{Semi-narrow class group}\label{subsection : semi-narrow class group}
Let $v$ be a real prime of $K$. As in \cite{BPT}, we define the following.

\begin{definition}\label{definition : real prime v1 v2 v3}
We say \textit{$v$ is ramified} (resp. \textit{unramified}) if $L_v \simeq \R \times \C$ (resp. $L_v \simeq \R \times \R \times \R$).
When $v$ is ramified, we denote by $\wt{v}$ the unique real prime above $v$. If $v$ is unramified, then we can write $F(x)=(x-\g_1)(x-\g_2)(x-\g_3)$ with $\g_i \in \R$ and $\g_1 < \g_2 < \g_3$. We fix an isomorphism $L_v \simeq \R\times \R \times \R$ given by $g(x) \mapsto (g(\g_1), g(\g_2), g(\g_3))$ and we denote by $\wt{v}$ (resp. $\wt{v}_2$ and $\wt{v}_3$)  the one corresponding to the first (resp. second and third) component.
\end{definition}

There is the canonical map $L^\times \to \sel {L_\R}$ induced by the sign map. More precisely, let $A$ (resp. $B$) be the set of the ramified (resp. unramified) real primes of $K$. Then we may identify $\sel {L_\R}$ with $\prod_{v \in A} \{ \pm 1\} \times \prod_{v \in B} (\{ \pm 1 \} \times \{ \pm 1 \} \times \{ \pm 1 \})$ and so we have
\begin{equation*}
\sign : L^\times \to \sel {L_\R}=\prod_{v \in A} \{ \pm 1\} \times \prod_{v \in B} (\{ \pm 1 \} \times \{ \pm 1 \} \times \{ \pm 1 \}).
\end{equation*}

Now, we consider two subgroups $\wt{V}$ and $\wt{V}'$ of $\sel {L_\R}$ as follows:
\begin{equation*}
\begin{split}
\wt{V}&:=\prod_{v \in A} \{ 1\} \times \prod_{v \in B} \{(1, 1, 1), (1, -1, -1)\},\\
\wt{V}'&:=\prod_{v \in A} \{\pm 1\} \times \prod_{v \in B} \{(1, 1, 1), (1, -1, -1), (-1, 1, 1), (-1, -1, -1)\}.
\end{split}
\end{equation*}
Also, we define
\begin{equation*}
P_L^\infty:=\sgn^{-1}(\wt{V}) \qa P_L^0 :=\sgn^{-1}(\wt{V}').
\end{equation*}
\begin{remark}
By \cite[Prop. 3.7]{BK77}, the group $\wt{V}$ is the one related to the archimedean local conditions for $\Sel_2(E/K)$. On the other hand, the group $\wt{V}'$ is chosen for the following reason. In Subsection \ref{subsection : m0minfty} we define $M_0$ and $M_\infty$, which are groups of quadratic characters of $L$ with ``archimedean local conditions'' corresponding to $P_L^0$ and $P_L^\infty$, respectively. 
It turns out that (see the proof of Lemma \ref{lemma: M0 and class group})
$$M_0 \cong  \Hom(\Gal(H_L^\infty/L), ~\mu_2) \qa M_\infty  \cong  \Hom(\Gal(H_L^0/L), ~\mu_2),$$
where $H_L^\infty$ and $H_L^0$ are the Hilbert class fields defined in Definition \ref{definition : variousclassgroup} below. 
Note the switch between the indexes ``$0$'' and ``$\infty$''. 
In particular,  Lemma \ref{lemma: M0 and class group} pins down the choice of $\wt{V}'$ from $\wt{V}$ for the computational purpose.
\end{remark}

Note that for any $\alpha \in L^\times$, we say it is \textit{totally positive}, denoted by $\a \gg 0$, if $w(\a)>0$ for all real primes $w$ of $L$. For simplicity, let $P_L:=L^\times$, and let $P_L^+:=\{ \a \in P_L : \a \gg 0 \}$. Then by definition we have
\begin{equation*}
P_L^+ \subset P_L^\infty \subset P_L^0 \subset P_L
\end{equation*}
and each quotient is an elementary abelian $2$-group. Moreover, it follows from the definition that
\begin{equation*}
\begin{split}
P_L^0 & =\{ \a \in P_L : \wt{v}_2(\a) \wt{v}_3(\a) >0  \text{ for all unramified real primes $v$ of $K$} \},\\
P_L^\infty & =\{ \a \in P_L : \wt{v}(\a) > 0 \text{ and } v(N(\a))>0 \text{ for all real primes $v$ of $K$} \}.
\end{split}
\end{equation*} 
Since the sign map is surjective, it is straightforward to check that  
\begin{equation*}
[P_L : P_L^0]=2^b, [P_L^0 : P_L^\infty]=2^{a+b}  \text{ and } [P_L^\infty : P_L^+]=2^b,
\end{equation*}
where $a=|A|$ and $b=|B|$.

\ms
\begin{definition}\label{definition : variousclassgroup}
Let $\star \in \{ \emptyset, +, 0, \infty \}$, and let $\cP_L^\star:=\{ (\a) \in \cF_L : \a \in P_L^\star \}$,
where $\cF_L$ is the group of fractional ideals  of $L$.\footnote{We use a capital Roman letter for the set of certain elements and the corresponding capital calligraphic letter for the set of principal fractional ideals generated by its elements.} Also, let $C^\star_L:=\cF_L/{\cP_L^\star}$ and let $H_L^\star$  be the class field of $L$ with respect to $C^\star_L$.
\end{definition}

\begin{remark}
The group $C^+_L$ is usually called the \textit{narrow class group} of $L$. If all the real primes of $K$ are ramified then $C^\infty_L=C^+_L$. Thus, we call $C^\infty_L$ the \textit{semi-narrow class group} of $L$, which is used in our title.
\end{remark}

Similarly as above, let $P_K$, $P_K^+$, $\cF_K$, $\cP_K$ and $\cP_K^+$ be the corresponding groups of $K$.  Also, let $C^\star_K:=\cF_K/{\cP_K^\star}$ and $H_K^\star$ for $\star \in \{ \emptyset, + \}$. Then we can easily check that $[P_K : P_K^+ ]=2^{a+b}$.

\ms
\subsection{The groups $M_0$ and $M_\infty$}
\label{subsection : m0minfty}
For $\star \in \{0, \infty \}$, let 
\begin{equation*}
M_\star:=\{ [\a] \in \sel L : ~\text{$L(\sqrt{\a})/L$ is unramified at all finite primes and } \a \in P_L^\star \}.
\end{equation*}

\begin{lemma}\label{lemma: M0 and class group}
We have
\begin{equation*}
M_0 \simeq C^\infty_L/{2C^\infty_L}\qa M_\infty \simeq C^0_L/{2C^0_L},
\end{equation*}
and hence $|M_0|=|C^\infty_L[2]|$ and $|M_\infty|=|C^0_L[2]|$.
\end{lemma}
\begin{proof}
By the class field theory, the field $H_L^\infty$ is the maximal abelian extension of $L$ satisfying
\begin{itemize}[--]
\item
it is unramified at all finite primes, and
\item
for any unramified real place $v$ of $K$, every quadratic subextension of $H_L^\infty/H_L$ is either unramified both at $\wt{v}_2$ and $\wt{v}_3$, or ramified both at $\wt{v}_2$ and $\wt{v}_3$.
\end{itemize}

Let $v$ be an unramified real prime of $K$ and $\a \in P_L^0$. Since $\wt{v}_2(\a)\wt{v}_3(\a)>0$, either $L(\sqrt{\a})$ is unramified both at $\wt{v}_2$ and $\wt{v}_3$, or ramified both at $\wt{v}_2$ and $\wt{v}_3$. Thus, for any $\alpha \in M_0$,
$L(\sqrt{\a})$ is a subfield of $H_L^\infty$. By Kummer theory, any quadratic subfield of $H_L^\infty$ is of the form $L(\sqrt{\a})$ for some $[\a] \in M_0$. Thus, we have an isomorphism
\begin{equation*}
g : M_0 \to \Hom(\Gal(H_L^\infty/L), ~\mu_2)
\end{equation*}
sending $[\a]$ to the character $\chi$ such that $(H_L^\infty)^{\ker (\chi)}=L(\sqrt{\a})$. Since $\Hom(\Gal(H_L^\infty/L), ~\mu_2) \simeq C^\infty_L/{2C^\infty_L}$ (not canonical though), the first isomorphism follows. By the same argument, the second also follows.

Since $C^\infty_L$ is finite, we have $|C_L^\infty/{2C_L^\infty}|=|C_L^\infty[2]|$ and similarly for $C_L^0$. This completes the proof.
\end{proof}

\ms
\subsection{The cardinality of $M_1$}
In this subsection, we prove the following, which implies the first equality of Theorem \ref{theorem: main theorem sizes of M1 and M2} by Lemma \ref{lemma: M0 and class group}.

\begin{proposition}\label{proposition: index of M0 Minfty M_1}
There is an isomorphism
\begin{equation*}
\frac{M_0}{M_1} \simeq C^+_K/{2C^+_K}
\end{equation*}
and hence $|M_1| = |M_0| \times |C^+_K[2]|^{-1}$.
\end{proposition}
\begin{proof} 
We claim that for any $[\a] \in M_0$ the extension field $K(\sqrt{N(\a)})$ is a subfield of $H_K^+$. 
This is proven in the proof of \cite[Lem. 5.2]{Sc94}, but we provide a complete proof for the convenience of the readers.

Let $[\a] \in M_0$. Since $L(\sqrt{\a})/L$ is unramified everywhere, $w(\a)$ is even for all finite primes $w$ of $L$. Thus, $v(N(\a))$ is also even for all finite primes $v$ of $K$ and hence $K(\sqrt{N(\a)})/K$ is unramified at all odd primes $v$ of $K$. Let $v$ be an even prime of $K$, and let $w$ be a prime of $L$ above $v$. Since $L(\sqrt{\a})/L$ is unramified at $w$, by Lemma \ref{lemma: unramified condition at 2} below and the weak approximation theorem we have $\a\b^2=x^2+4y$ for some $\b \in L^\times$, $x \in \cO_{L}^\times$ and $y \in \cO_{L}$. Thus,
\begin{equation*}
 N(\a)\cdot N(\b)^2=N(\a \b^2)=N(x)^2+4y' \text{ for some } y' \in \cO_K.
 \end{equation*}
By Lemma \ref{lemma: unramified condition at 2}, $K(\sqrt{N(\a)})=K(\sqrt{N(\a\b^2)})$ is unramified at $v$.
This proves the claim. 

As a result, we have a group homomorphism
\begin{equation*}
f : M_0 \to \Hom(\Gal(H_K^+/K), \mu_2)
\end{equation*}
sending $[\a] \in M_0$ to the character $\chi$ such that $(H_K^+)^{\ker (\chi)}=K(\sqrt{N(\a)})$. We claim that $f$ is surjective. Let $\chi \in \Hom(\Gal(H_K^+/K), \mu_2)$ and let $K'=(H_K^+)^{\ker (\chi)}$. Then there is an element $\a \in K^\times$ such that $K' \simeq K(\sqrt{\a})$. 
Since $K(\sqrt{\a})/K$ is unramified at all finite primes, so is $L(\sqrt{\a})/L$. Since $\wt{v}_2(\a)=\wt{v}_3(\a)$ for any unramified real primes $v$ of $K$ (as $\a \in K^\times$), we have $\a \in P_L^0$ and hence $[\a] \in M_0$. 
Since $L/K$ is of degree $3$ and $\a$ is chosen in $K^\times$, we have $N(\a)=\a^3$. Thus, we have $K(\sqrt{N(\a)})=K(\sqrt{\a^3})=K(\sqrt{\a})$, which is isomorphic to $K'$. Hence $f([\a])=\chi$, as claimed.

To prove the first assertion, it suffices to show that $\ker (f)=M_1$. It is easy to see that $M_1 \subset \ker(f)$.
Conversely, suppose that $[\a] \in \ker (f)$ for some $\a \in P_L^0$, i.e., $N(\a)$ is a square. Then we have $N(\a) \gg 0$. Since $\a \in P_L^0$ and $N(\a) \gg 0$, we have $\wt{v}(\a)>0$ for all real primes $v$ of $K$ as well. Thus, we have $\a \in P_L^\infty$ and $[\a] \in M_1$, as desired. This proves the first assertion. The second follows from the finiteness of $C^+_K$.
\end{proof}
\begin{remark}
Similarly, we can prove $M_\infty/{M_1} \simeq C_K/{2C_K}$ and hence $[M_0 : M_\infty]=\frac{|C^+_K[2]|}{|C_K[2]|}$. 
\end{remark}

\begin{lemma}\label{lemma: unramified condition at 2}
Let $H/{\Q_2}$ be a finite extension. Then for $\a \in \cO_H^\times$, the extension $H(\sqrt{\a})/H$ is unramified if and only if
$\a \equiv u^2 \pmod {4\cO_H}$ for some $u \in \cO_{H}^\times$.
\end{lemma}
\begin{proof}
This is elementary, for example, see \cite[Prop. 4.8]{DV18}.
\end{proof}

\ms
\subsection{The cardinality of $M_2$}
In this subsection, we prove the second equality of Theorem \ref{theorem: main theorem sizes of M1 and M2}. 
In order to do it, we use two natural maps\footnote{The map $\g$ is well-known, for example in \cite[(3.4)]{DV18}, \cite[Lem. 2.17]{Li19} and \cite[Lem. 2.13]{BPT}.}
\begin{equation*}
\g : M_2 \to C_L[2] \qa \pi : C^\infty_L[2] \to C_L[2].
\end{equation*}
By computing the precise kernels of two maps, and comparing their images, we have the following. 
\begin{proposition}\label{proposition: index of M2 M0}
We have 
\begin{equation*}
\frac{|M_2|}{|C_L^\infty[2]|}=\frac{2^{[K:\Q]}}{|C^+_K[2]|}.
\end{equation*}
\end{proposition}
We remark that the idea of using the maps $\gamma$ and $\pi$ is already appeared in \cite{BPT} (under the assumption that $K$ has an odd narrow class number) and we closely follow their strategy. Our contribution is to verify that it works for any number field $K$ (and we precisely compute the ratio of the images of two maps in Step 1 below).
For the convenience of the readers, we provide a complete proof.
We use the same notation as in Section \ref{subsection : semi-narrow class group}.

\ms
We prove the proposition by four steps. Before proceeding, we define precisely two morphisms $\pi$ and $\g$.

\ms
First, we consider the map $\wt{\pi} : C^\infty_L \to C_L$ sending $I \pmod {\cP_L^\infty}$ to $I \pmod {\cP_L}$ for any $I \in \cF_L$. Let $\pi$ be the restriction of $\wt{\pi}$ to $C^\infty_L[2]$. Since the kernel of $\wt{\pi}$ is $\frac{\cP_L}{\cP_L^\infty}$, which is an elementary abelian $2$-group, we have an exact sequence
\begin{equation}
\label{equation : ker pi}
\xymatrix{
0 \ar[r] & \frac{\cP_L}{\cP_L^\infty} \ar[r] & C^\infty_L[2] \ar[r]^-{\pi} & C_L[2].
}
\end{equation}
Similarly, we have a map $\pi_K : C^+_K[2] \to C_K[2]$. It can be easily checked that $\ker (\pi_K) = \frac{\cP_K}{\cP_K^+}$ and $|\ker(\pi_K)|=\frac{|C^+_K|}{|C_K|}$. 

\ms
Next, we construct a surjective map $\wt{f}$ from a subset of $P_L^\infty$ to $C^\infty_L[2]$ as follows: Since any element $[I] \in C^\infty_L[2]$ satisfies $I^2 \in \cP_L^\infty$, so we can find an element $\a \in P_L^\infty$ such that $(\a)=I^2$. 
So for $\a \in P_L^\infty$ with $(\a)=I^2$ for some $I \in \cF_L$, we set $\wt{f}(\a) :=I \pmod {\cP_L^\infty}$, which is well-defined. This map induces a surjective map $f : M_L^\infty \to C^\infty_L[2]$, where 
\begin{equation*}
M_L^\infty:=\{ [\a] \in {P_L^\infty}/{(P_L^\infty)^2} : (\a)=I^2 \text{ for some $I \in \cF_L$} \}.
\end{equation*}
Similarly, we have a surjective map $f_K : M_K^+ \to C^+_K[2]$, where
\begin{equation*}
M_K^+:=\{ [a] \in {P_K^+}/{(P_K^+)^2} : (a)=J^2 \text{ for some $J \in \cF_K$} \}.
\end{equation*}

\ms
Then, we consider the composition $\pi \circ f : M_L^\infty \to C_L[2]$. This map factors through 
\begin{equation*}
M_L := \{ [\a] \in \sel L : (\a)=I^2 \text{ for some $I \in \cF_L$} \qqa ~\a \in P_L^\infty\}
\end{equation*}
and let $\g_L : M_L \to C_L[2]$ be the map induced by $\pi \circ f$. Indeed, if $[\a] \in M_L$ and write $(\a)=I^2$, then $\g_L([\a])=I \pmod {\cP_L}$. Similarly, we have a map $\g_K : M_K \to C_K[2]$, where
\begin{equation*}
M_K:=\{ [a] \in \sel K : (a)=J^2 \text{ for some $J \in \cF_K$} \qqa ~a \in P_K^+\}.
\end{equation*}
We then define the map $\g$ by the restriction of $\g_L$ to $M_2$, i.e., $\g := \g_L |_{M_2} : M_2 \to C_L[2]$. 

\ms
Lastly, we have the map $N: \sel {L} \to \sel K$ induced by the norm map. It induces well-defined maps $g_1 : M_L^\infty \to M_K^+$ and $g_2 : M_L \to M_K$ sending $[\a]$ to $N([\a])$.

\ms
In summary, we have a commutative diagram
\begin{equation*}
\xyv{0.7}
\xyh{4}
\xymatrix{
M_L^\infty \ar@{->>}[rd]\ar@{->>}[rr]^{f} \ar[dd]^>>>>{g_1} & & C^\infty_L[2] \ar[rd]^-{\pi} \ar@{-->}[dd]&  \\
& M_L \ar@{->>}[dd]^>>>>{g_2} \ar[rr]^<<<<<<<<<<<<{\g_L} & &   C_L[2] \ar@{-->}[dd] \\
M_K^+ \ar@{->>}[rd]\ar@{->>}[rr]^>>>>>>>>{f_K} & & C^+_K[2] \ar[rd]^-{\pi_K} & \\
& M_K \ar[rr]^<<<<<<<<<<<<{\g_K} & & C_K[2].\\
}
\end{equation*}

\ms
\noindent
\textbf{$\bullet$ Step 1: Comparison of the images}. Since $f$ is surjective, we have $\im (\g_L)=\im (\pi)$ and hence $\im (\g) \subset \im(\pi)$. Moreover, we assert the following.
\begin{proposition}\label{proposition: poor inequality} 
We have 
\begin{equation*}
\frac{\im (\pi)}{\im (\g)} \simeq \im(\pi_K) \qa \frac{|\im (\pi)|}{|\im (\g)|}=\frac{|C^+_K[2]| \times |C_K|}{|C^+_K|}.
\end{equation*}
\end{proposition}
\begin{proof}
We first claim that the map $g_2$ induces an isomorphism
\begin{equation*}
\frac{M_L}{M_2 \cdot \ker (\g_L)} \simeq \frac{M_K}{\ker (\g_K)}.
\end{equation*}
By definition, we have $\ker (\g_\star) = \{[\a] \in M_\star : (\a)=(\b)^2 \text{ for some } \b \in P_\star \}$ for $\star \in \{K, L\}$.
Let $h : M_K \to M_L$ be the map sending $[a]$ to $[a]$. Then $g_2 \circ h$ is the identity (because $[L : K]=3$) and the kernel of $g_2$ is $M_2$. Thus, to prove the claim, it suffices to show that $g_2(\ker(\g_L))=\ker(\g_K)$. Indeed, let $[\a] \in \ker(\g_L)$. Then $\a=u \cdot \b^2$ for some $u\in \cO_L^\times$ and $\b \in P_L$. Since $N(u) \in \cO_K^\times$, $N(\b) \in P_K$ and $N(\a)=N(u) \cdot (N(\b))^2$, we have $g_2([\a])=N([\a])=[N(\a)] \in \ker(\g_K)$. Conversely, if $[\b] \in \ker(\g_K)$ then it is easy to see that $g(h([\b]))=[\b]$ and $h([\b]) \in \ker(\g_L)$. This proves the claim.

Next, we prove the proposition. Note that $\im(\pi)=\im(\g_L)$ and similarly, $\im (\pi_K)=\im (\g_K)$. Since the kernel of the composition
\begin{equation*}
\xymatrix{
M_L \ar@{->>}[r]^-{\g_L}  & \im (\g_L)=\im (\pi) \ar@{->>}[r] & \frac{\im (\pi)}{\im (\g)}
}
\end{equation*}
is $M_2 \cdot \ker (\g_L)$ and $\frac{M_K}{\ker (\g_K)} \simeq \im(\g_K)=\im (\pi_K)$, the first assertion follows.
Since $|\ker(\pi_K)| \times |\im (\pi_K)|=|C^+_K[2]|$ and $|\ker (\pi_K)|= \frac{|C^+_K|}{|C_K|}$, we obtain the result.
\end{proof}

\ms
\noindent
\textbf{$\bullet$ Step 2: Computation of the kernel of $\g$}. Recall that $A$ (resp. $B$) is the set of all ramified (resp. unramified) real primes of $K$, and $a=|A|$ (resp. $b=|B|$). Also, $C$ is the set of complex primes of $K$, and $c=|C|$. Note that $[K:\Q]=a+b+2c$ and the number of real (resp. complex) primes of $L$ is $a+3b$ (resp. $a+3c$). Note also that there is the canonical map
\begin{equation*}
\sign : L^\times \to \sel {L_\R}= \prod_{v \in A} \{ \pm 1\} \times \prod_{v \in B} (\{ \pm 1 \} \times \{ \pm 1 \} \times \{ \pm 1 \})
\end{equation*}
which we often regard as the map from $\sel {L}$ (or its subgroups). 
For simplicity, let
\begin{equation*}
\wt{W} = \prod_{v \in A} \{ 1 \} \times \prod_{v \in B} \{(1, 1, 1), (1, -1, -1), (-1, 1, -1), (-1, -1, 1) \} \subset \sel {L_\R}.
\end{equation*}

\ms
First, we prove the following.
\begin{lemma}\label{lemma: isom1}
We have
\begin{equation*}
\ker (\g) = (\sel {\cO_L})_\ns \cap M_2.
\end{equation*}
\end{lemma}
\begin{proof}
Let $[\a] \in \ker (\g)$. 
If we write $I^2=(\a)$, then $I$ is principal by definition, so $I=(\b)$ for some $\b \in L^\times$. In other words, $(\a)=(\b^2)$ and hence there is a unit $u \in \cO_L^\times$ such that $\a=\b^2 u$. Note that $[\a]=[u]$ and so it suffices to show that $N(u)$ is a square. Since $[\a] \in M_2$, $N(\a)=c^2$ for some $c \in K^\times$. Hence, $N(u)=N(\a)\times N(\b)^{-2}=(c N(\b)^{-1})^2$ is a square, as desired.

Conversely, if $[\a] \in (\sel {\cO_L})_\ns \cap M_2$, then we have $(\a)=\cO_L=(\cO_L)^2$ (as $\a \in \cO_L^\times$). Thus, $I=\cO_L=(1)$ is principal and $[\a] \in \ker(\g)$. 
\end{proof}
Note that if $N(\a)$ is a square then $\sign(\a) \in \wt{W}$. 
Note also that $\sign(\a) \in \wt{V}$ if and only if $\a \in P_L^\infty$ by definition. Thus, $\sign^{-1}(\wt{V}) \cap (\sel {\cO_L})_\ns \subset M_2$ and so we have the following.
\begin{lemma}\label{lemma: isom2}
The kernel of $\g$ is isomorphic to that of the composition 
\begin{equation*}
\xymatrix{
(\sel {\cO_L})_\ns \ar[r]^-{\sign} &\sign((\sel {\cO_L})_\ns) \ar@{->>}[r] &\frac{\sign((\sel {\cO_L})_\ns)}{\sign((\sel {\cO_L})_\ns) \cap \wt{V}}.
}
\end{equation*}
\end{lemma}
\begin{proof}
By Lemma \ref{lemma: isom1}, we have $\ker (\g)=(\sel {\cO_L})_\ns \cap M_2$ and hence the result follows.
\end{proof}
By the second isomorphism theorem, we have the following.
\begin{lemma}\label{lemma: isom3}
We have
\begin{equation*}
\frac{\sign((\sel {\cO_L})_\ns)}{\sign((\sel {\cO_L})_\ns) \cap \wt{V}}
\simeq \frac{\sign((\sel {\cO_L})_\ns)\cdot \wt{V}}{\wt{V}}.
\end{equation*}
\end{lemma}
Finally, we have the following.
\begin{lemma}\label{lemma: isom4}
There is an isomorphism
\begin{equation*}
\sign((\sel {\cO_L})_\ns) \cdot \wt{V} \simeq \sign(\cO_L^\times)\cdot \wt{V} /\sign(\cO_K^\times).
\end{equation*}
\end{lemma}
\begin{proof}
Let $f$ be the map from $\sign((\sel {\cO_L})_\ns)$  to  $\sign(\cO_L^\times)/\sign(\cO_K^\times)$ defined by $f(\sign([\a]))=\sign(\a)\cdot \sign(\cO_K^\times)$ for any $[\a] \in  (\sel {\cO_L})_\ns$. We claim that this map is an isomorphism. Let $\a \in \cO_L^\times$. Since 
$\sign(\a N(\a))=\sign(\a) \cdot \sign(N(\a))$ and $N(\a) \in \cO_K^\times$, 
we have
\begin{equation*}
\sign(\a)\cdot \sign(\cO_K^\times) = \sign(\a N(\a)) \cdot \sign(\cO_K^\times) = f(\sign([\a N(\a)])).
\end{equation*}
Since $N(\a N(\a))=N(\a)^4$, we have $[\a N(\a)] \in (\sel {\cO_L})_\ns$ and hence $f$ is surjective. Next, since
\begin{equation*}
\sign(\cO_K^\times) \subset \prod_{v\in A} \{ \pm 1 \} \times \prod_{v\in B} \{(1, 1, 1), (-1, -1, -1) \},
\end{equation*}
the intersection of $\wt{W}$ and $\sign(\cO_K^\times)$ is trivial. Since $\sign(\a) \in \wt{W}$ for any $\a \in (\sel L)_\ns$, $f$ is injective as claimed. 
By multiplying on both sides by $\wt{V}$, we get the desired isomorphism because $\sign(\cO_K^\times) \cap \wt{V}$ is also trivial.
\end{proof}

Combining all the results above, we have the following.
\begin{proposition}\label{proposition: the size of ker gamma}
We have
\begin{equation*}
|\ker (\g)|=\frac{|(\sel {\cO_L})_\ns| \times |\wt{V}| \times |\sign(\cO_K^\times)|}{|\sign(\cO_L^\times)\cdot \wt{V}|}.
\end{equation*}
\end{proposition}

\ms
\noindent
\textbf{$\bullet$ Step 3: Computation of the kernel of $\pi$}. As shown in \eqref{equation : ker pi}, we have $\ker (\pi) \simeq \frac{\cP_L}{\cP_L^\infty}$. Using the sign map, we obtain the following.
\begin{lemma}\label{lemma: isom5}
We have
\begin{equation*}
 \frac{\cP_L}{\cP_L^\infty} \simeq \frac{L^\times}{\cO_L^\times \cdot \sign^{-1}(\wt{V})}.
\end{equation*}
\end{lemma}
\begin{proof}
Let $f$ be the composition
\begin{equation*}
\xyh{4}
\xymatrix{
\frac{L^\times \ar[r]}{\cO_L^\times}  \ar[r]_-{\a \mapsto (\a)}^-\sim & \cP_L \ar[r] & \frac{\cP_L}{\cP_L^\infty},
}
\end{equation*}
which is clearly surjective. It is straightforward to check that $\ker (f)=\sign^{-1}(\wt{V}) \cdot \cO_L^\times/{\cO_L^\times}$, which completes the proof.
\end{proof}
Again, by the sign map we have the following.
\begin{lemma}\label{lemma: isom6}
The sign map induces an isomorphism
\begin{equation*}
 \frac{L^\times}{\cO_L^\times \cdot \sign^{-1}(\wt{V})} \overset{\sim}{\longrightarrow}
\frac{\sign(L^\times)}{\sign(\cO_L^\times) \cdot \wt{V}}.
\end{equation*}
\end{lemma}
\begin{proof}
It suffices to show that if $\sgn(\a) \in \sign(\cO_L^\times)\cdot \wt{V}$ for some $\a \in L^\times$, then $\a \in \cO_L^\times \cdot \sign^{-1}(\wt{V})$.
By the assumption,  there is $\b \in \cO_L^\times$ such that $\sgn(\a) \in \sgn(\b)\cdot \wt{V}$, or equivalently, $\sign(\a/\b) \in \wt{V}$. Thus, $\a/\b \in \sign^{-1}(\wt{V})$ and hence $\a \in \b \cdot \sign^{-1}(\wt{V}) \subset \cO_L^\times \cdot \sign^{-1}(\wt{V})$, as desired.
\end{proof}

Combining two results above, we have the following.
\begin{proposition}\label{proposition: the size of ker pi}
We have
\begin{equation*}
|\ker (\pi)|=\frac{|\sign(L^\times)|}{|\sign(\cO_L^\times)\cdot \wt{V}|}.
\end{equation*}
\end{proposition}

\ms
\noindent
\textbf{$\bullet$ Step 4: Proof of Proposition \ref{proposition: index of M2 M0}}.
Since 
\begin{equation*}
|M_2|=|\ker (\g)|\times |\im(\g)| \qa |C^\infty_L[2]|=|\ker (\pi)| \times |\im(\pi)|,
\end{equation*}
Propositions \ref{proposition: poor inequality}, \ref{proposition: the size of ker gamma} and \ref{proposition: the size of ker pi} we have
\begin{equation*}
\frac{|M_2|}{|C_L^\infty[2]|}=\frac{|(\sel {\cO_L})_\ns| \times |\wt{V}| \times |\sign(\cO_K^\times)|\times |C^+_K|}{|\sign(L^\times)|\times |C^+_K[2]|\times |C_K|}.
\end{equation*}
By the lemma below, we obtain the result. \qed
\begin{lemma}\label{lemma: basic quantities}
We have the following.
\begin{enumerate}
\item
$[K:\Q]=a+b+2c$ and $|\wt{V}|=2^b$.
\item
$|\sign(K^\times)|=2^{a+b}$ and $|\sign(L^\times)|=2^{a+3b}$.
\item
$|\sign(\cO_K^\times)|=2^{a+b} \times |C_K| \times |C^+_K|^{-1}$.
\item
$|\sel {\cO_K}|=2^{a+b+c}$ and $|\sel {\cO_L}|=2^{2a+3b+3c}$.
\item
$|(\sel {\cO_L})_\ns|=2^{a+2b+2c}$.
\end{enumerate}
\begin{proof}
The first assertion is obvious. Note that the sign map is surjective by the weak approximation theorem. Thus, the second one follows.
Next, consider the exact sequence (cf. Example 1.8 (b) of Chapter V in \cite{Mi13})
\begin{equation*}
\xymatrix{
0 \ar[r] & \cO_K^\times/{(\cO_K^\times)_+} \ar[r] & {P_K}/{P_K^+}  \ar[r] & C^+_K \ar[r] & C_K \ar[r] & 0,
}
\end{equation*}
where $(\cO_K^\times)_+=\cO_K^\times \cap P_K^+$. Since the sign map induces an isomorphism $P_K/{P_K^+} \simeq \sign (K^\times)$ and $\cO_K^\times/{(\cO_K^\times)_+} \simeq \sign (\cO_K^\times)$, the third one follows.
Then, by Dirichlet's unit theorem for any number field $H$ we have 
$|\sel {\cO_H}|=2\times 2^{r_1+r_2-1}=2^{r_1+r_2}$, where $r_1$ (resp. $r_2$) denotes the number of real primes (resp. complex) primes. Thus, the fourth one follows.
Lastly, note that the norm map $N: \sel {\cO_L} \to \sel {\cO_K}$ is surjective because $[L:K]=3$. Since $(\sel {\cO_L})_\ns$ is the kernel of the norm map, the last one follows by the fourth assertion.
\end{proof}
\end{lemma}

\ms
\section{The local conditions}\label{section: local conditions}
As before, let $K$ be a number field and let $F(x)$ be an irreducible cubic polynomial in $\cO_K[x]$.
Also, let $L=K[x]/{(F(x))}$ be a cubic extension of $K$. 

\subsection{Infinite primes}\label{section: 3.1}
Let $v$ be an infinite prime of $K$. Following the notation in Definition \ref{definition : real prime v1 v2 v3}, we define $M_{i, v} \subset \sel {L_v}$ as follows: Let
\begin{equation*}
M_{1, v}=M_{2, v}:=\begin{cases}
\{([1], [1]) \}  & \text{ if $v$ is real and ramified},\\
\{([1], [1], [1]), ([1], [-1], [-1])\} & \text{ if $v$ is real and unramified},\\
\{([1], [1], [1]) \} & \text{ if $v$ is complex}.
\end{cases}
\end{equation*}
By \cite[Prop. 3.7]{BK77}, these coincide with the local condition $\im(\d_{K_v})$ of $\Sel_2(E/K)$ at $v$. 

\ms
\subsection{Finite primes}\label{subsection: finite primes}
Before proceeding, we fix notations.

Let $v$ be a finite prime of $K$, $\cO_{K_v}$ the ring of integers of $K_v$, $\pi$ a uniformizer and $k=\cO_{K_v}/{(\pi)}$ the residue field of $K_v$.
Also, let $\{w_1, \dots, w_n \}$ ($1\leq n \leq 3$) be the primes of $L$ above $v$, $\cO_{L_v}$ the integral closure of $\cO_{K_v}$ in $L_v$. 
For any element $\a \in L$, let $\a_v$ (resp. $\a_w$) be the image of $\a$ by the embedding $\iota_v : L \inj L_v$ (resp. $\iota_w : L \inj L_w$).
From now on, we fix an isomorphism $\phi_v : L_v \simeq L_{w_1} \times \cdots \times L_{w_n}$ which gives rise to a commutative diagram
\begin{equation*}
\xymatrix{
L \ar[d]_-{\iota_v} \ar@/^1pc/[dr]^-{\prod_{i=1}^n \iota_{w_i}} & \\
L_v \ar[r]_-{\phi_v} & L_{w_1} \times \cdots \times L_{w_n}.
}
\end{equation*}
Under the map $\phi_v$ we have natural isomorphisms
\begin{equation*}
\sel {L_v} \simeq \sel {L_{w_1}} \times \cdots \times \sel {L_{w_n}}
\end{equation*}
and
\begin{equation*}
\sel {\cO_{L_v}} \simeq \sel {\cO_{L_{w_1}}} \times \cdots \times \sel {\cO_{L_{w_n}}}.
\end{equation*}

\ms
First, let $\a \in L^\times$. If $w$ is odd, then it is easy to see that
\begin{equation*}
L_w(\sqrt{\a_w})/{L_w} \text{ is unramified} \iff w(\a_w) \in 2\Z \iff \a_w \in \cO_{L_w}^\times \text{ modulo squares}.
\end{equation*}
Also, if $w$ is even then by Lemma \ref{lemma: unramified condition at 2}
\begin{equation*}
L_w(\sqrt{\a_w})/{L_w} \text{ is unramified}  \iff \a_w \in 1+4\cO_{L_w} \text{ modulo squares}.
\end{equation*}
These conditions are equivalent to the assertion $[\alpha_w] \in M_{0,w}$ where
\begin{equation*}
M_{0,w} :=
\begin{cases}
\sel {\cO_{L_w}} & \text{ if $w$ is odd},\\
\{[1], ~[\boxtimes'] \} & \text{ if $w$ is even}.
\end{cases} 
\end{equation*}
Here $\boxtimes' \in 1+4\cO_{L_w}$ is chosen so that $L_w(\sqrt{\boxtimes'})$ is a unique unramified quadratic extension of $L_w$. Similarly, for a finite prime $v$ of $K$ below $w$, there is an element $\boxtimes \in 1+4\cO_{K_v}$ such that $K_v(\sqrt{\boxtimes})/K_v$ is the unramified quadratic extension, which is unique modulo squares.
The following is useful in the sequel.
\begin{lemma}\label{lemma : norm relation at 2}
Let $v$ be an even prime of $K$, and $w$ a prime of $L$ above $v$. Also, let
\begin{equation*}
\text{Nm} : \sel {\cO_{L_w}} \to \sel {\cO_{K_v}}
\end{equation*}
be the map induced by the norm map $N:L_w^\times \to K_v^\times$. If the ramification degree of $L_w/{K_v}$ is odd then we have $\text{Nm}([\boxtimes'])=[\boxtimes]$. If $L_w$ is a ramified quadratic extension of $K_v$, then $\text{Nm}([\boxtimes '])=[1]$. 
\end{lemma}
\begin{proof}
Note that $\boxtimes ' \in 1+4\cO_{L_w}$ is not a square. 
By \cite[Lem. 1.10]{BPT}, $N(\boxtimes ' ) \in 1+4\cO_{K_v}$ is a square (resp. not a square) if the ramification index of $L_w/K_v$ is even (resp. odd). Thus, the result follows.
\end{proof}

\ms
Now we study the local condition $M_{i,v}$ of $M_i$ defined in Section \ref{section: introduction}.
If $v$ is odd then
\begin{equation*}
M_{1, v}=M_{2, v}=(\sel{\cO_{L_v}})_\ns.
\end{equation*}
Thus, we henceforth assume that $v$ is an \textbf{even} prime of $K$. 
It follows form the definition of $M_{1,v}$ that $|M_{1,v}| = |E(K_v)[2]|$. 
So we divide into three cases.

\noindent
\textbf{Case 1}. $|E(K_v)[2]|=1$. Then there is a unique prime $w$ of $L$ and $\phi_v : L_v \simeq L_w$, and we have 
\begin{equation*}
\begin{split}
M_{1, v}&=\{[1] \},\\
M_{2, v}&=(\sel {\cO_{L_v}})_\ns.
\end{split}
\end{equation*}

\noindent
\textbf{Case 2}. $|E(K_v)[2]|=2$. There is a unique prime $w$ of $L$ such that $L_w \simeq K_v(\sqrt{\D})$ is a quadratic extension of $K_v$, where $\D$ is the discriminant of $E$, and $\phi_v : L_v \simeq K_v \times L_w$. By Lemma \ref{lemma : norm relation at 2} and the norm condition we have
\begin{equation*}
\begin{split}
M_{1, v}&=\begin{cases}
\{ ([1], [1]), ([\boxtimes], [\boxtimes']) \} & \text{ if  $L_w/{K_v}$ is unramified},\\
\{([1], [1]), ([1], [\boxtimes']) \} & \text{ if  $L_w/{K_v}$ is ramified, and}
\end{cases}\\
M_{2, v}&=\{ (\text{Nm}([\a_w]), [\a_w]) : \a_w \in \cO_{L_w}^\times \}.
\end{split}
\end{equation*}

\noindent
\textbf{Case 3}. $|E(K_v)[2]|=4$. In this case, we have $\phi_v : L_v \simeq K_v \times K_v \times K_v$. 
Also, we have
\begin{equation*}
\begin{split}
M_{1, v}&=\{([1], [1], [1]), ([1], [\boxtimes], [\boxtimes]), ([\boxtimes], [1], [\boxtimes]), ([\boxtimes], [\boxtimes], [1]) \},\\
M_{2, v}&=\{([a], [b], [ab]) : a, b \in \cO_{K_v}^\times \}.
\end{split}
\end{equation*}

\ms
\section{Criteria for niceness}\label{section: criteria for niceness}
For an elliptic curve $E$ over a number field $K$ given in the form $y^2=F(x)$ with $F(x) \in \cO_K[x]$, we hope to find criteria when $E$ is nice at a finite prime $v$ of $K$.
Let
\begin{equation}\label{equation : minimal model}
y^2+a_1 xy + a_3 y = x^3 + a_2 x^2 + a_4 + a_6 \text{ with  } a_i \in \cO_{K_v}
\end{equation}
be a minimal Weierstrass equation of $E$ over $K_v$. Then there is a filtration 
\begin{equation*}
 E_1(K_v) \subset E_0(K_v) \subset E(K_v),
 \end{equation*}
where $E_0(K_v)$ (resp. $E_1(K_v)$) is the subgroup of points of $E(K_v)$ whose reduction is non-singular (resp. trivial) (cf. \cite[Ch. VII, Prop. 2.1]{Si09}). 

\ms
First, let $v$ be an odd prime of $K$. Then we have $|\im(\d_{K_v})|=|M_{i, v}|=|E(K_v)[2]|$ (cf. \cite[Lem. 3.1]{BK77}) and therefore $E/{K_v}$ is nice if and only if it is lower (or upper) nice. Recall that $D$ denotes the discriminant of $F$. 
\begin{theorem}\label{theorem: odd niceness}
If $v$ is odd, then $E$ is nice at $v$ if one of the following holds.
\begin{enumerate}
\item
$|E(K_v)[2]|=1$.
\item
$v(D)\leq 1$.
\item
$[E(K_v) : E_0(K_v)]$ is odd.
\end{enumerate}
\end{theorem}
\begin{proof}
The first case is trivial because $M_{i, v}=\im(\d_{K_v})=\{[1]\}$. 
For the second case, see Proposition \ref{proposition: valuation is at most 1} below, which works without assuming that $v$ is odd. Thus, the second one follows.
The third one follows from Corollary 3.3 (and Remark) in \cite{BK77}.
\end{proof}

\begin{remark}
If $E$ has split multiplicative reduction at $v$ and $[E(K_v) : E_0(K_v)]$ is even, then $E$ is not nice at $v$. (This can be proved by \cite[Prop. 4.1]{BK77}.) 
\end{remark}

\ms
For the rest of this section, we assume that $v$ is an \textbf{even} prime of $K$ unless otherwise stated. 
For simplicity, let $d=[K_v:\Q_2]$, $e=v(2)$ the ramification index of $K_v$ over $\Q_2$, $\pi$ a uniformizer of $\cO_{K_v}$ and $k=\cO_{K_v}/{(\pi)}$ the residue field. Also, let $\wt{E}$ be the reduction of $E$ modulo $(\pi)$.

\begin{lemma}\label{lemma: selmer size at 2}
We have 
\begin{equation*}
|M_{1, v}|=|E(K_v)[2]| \qa \frac{|\im (\d_{K_v})|}{|M_{1, v}|}=\frac{|M_{2, v}|}{|\im (\d_{K_v})|}=[\cO_{K_v} : 2 \cO_{K_v}]=2^d.
\end{equation*}
\end{lemma}
\begin{proof}
This follows from the discussion in Section \ref{section: local conditions} and \cite[Lem. 3.1]{BK77}.
\end{proof}

One easy criterion is the following. 
\begin{proposition}\label{proposition: cubic extension}
Suppose that $|E(K_v)[2]|=1$. Then $E$ is nice at $v$.
\end{proposition}
\begin{proof}
It suffices to show that $E$ is upper nice at $v$, or equivalently, the valuation of $\d_{K_v}([P])$ for any $P \in E(K_v)$ is even. Let $P\in E(K_v)$. Then the valuation of the norm of $\d_{K_v}([P])$ is even because $y(P)^2=F(x(P))=N(\d_{K_v}([P]))$. Since the degree $[L_w : K_v]$ is $3$, the valuation of $\d_{K_v}([P])$ is also even. This completes the proof.
\end{proof}

Another criterion motivated by \cite{Li19} is the following.
\begin{proposition}\label{proposition: valuation is at most 1}
Let $D$ be the discriminant of $F$. If $v(D)\leq 1$, then $E$ is nice at $v$.
\end{proposition}
\begin{proof}
By Lemma \ref{lemma: etale squarefree discriminant} below, $E$ satisfies the condition $(\dagger.\text{ii})$ in \cite[Def. 1.6]{BPT}. Thus, the result follows by Theorem 1.11 of \textit{op. cit.}
\end{proof}
\begin{lemma}\label{lemma: etale squarefree discriminant}
Let $F(x) \in \cO_{K_v}[x]$ be a monic and separable polynomial with discriminant $D$. If $v(D)\leq 1$, then  the ring of integers of $K_v[x]/(F(x))$ is $\cO_{K_v}[x]/(F(x))$.
\end{lemma}
\begin{proof}
Let $F(x)=\prod_{i=1}^n F_i(x)$ with $F_i(x) \in \cO_{K_v}[x]$ monic, separable and irreducible. 
Note that $K_v[x]/{(F(x))} \simeq \prod_{i=1}^n K_v[x]/{(F_i(x))}$. Thus, it suffices to show that
\begin{enumerate}
\item
the ring of integers of $K_v[x]/{(F_i(x))}$ is $\cO_{K_v}[x]/{(F_i(x))}$; and
\item
there is an isomorphism:
\begin{equation*}
\cO_{K_v}[x]/{(F(x))} \simeq \prod_{i=1}^n \cO_{K_v}[x]/{(F_i(x))}.
\end{equation*}
\end{enumerate}

By definition, we have $\prod_{i=1}^n \textnormal{disc}(F_i) \mid D$, where $\textnormal{disc}(F_i)$ is the discriminant of $F_i$.
Since $v(D)\leq 1$, we may assume that $v(\textnormal{disc}(F_i))=0$ for all $1\leq i \leq n-1$ and $v(\textnormal{disc}(F_n))\leq 1$.

{\it Proof of }(1). Since $F_i$ are irreducible, we have
\begin{equation*}
 \textnormal{disc}(F_i) =\textnormal{disc}(R_i) \cdot \left[R_i : \cO_{K_v}[x]/{(F_i(x))}\right]^2,
\end{equation*}
where $R_i$ is the ring of integers of $K_v[x]/{(F_i(x))}$. Since $v(\textnormal{disc}(F_i)) \leq 1$ for all $i$, we have $R_i=\cO_{K_v}[x]/{(F_i(x))}$, as desired. \qed

{\it Proof of }(2). If $n=1$, it is vacuous, so we assume that $n\geq 2$. Let $G(x)=\prod_{i=2}^n F_i(x) \in \cO_{K_v}[x]$ so that $F(x)=F_1(x) \cdot G(x)$. Also, let $\alpha$ be a root of $F_1(x)$. Since $F_1(x)$ is monic and irreducible, $F_1(x)$ is the minimal polynomial of $\alpha$. Let $\cO_1:=\cO_{K_v}[\alpha] \simeq \cO_{K_v}[x]/{(F_1(x))}$, and let $w$ be the (normalized) valuation of $\cO_1$. Since the discriminant of $F_1(x)$ is a unit in $\cO_{K_v}$, $\cO_1/\cO_{K_v}$ is unramified and so $w(D)=v(D)\leq 1$. Also since $G(\alpha)^2$ divides $D$,\footnote{For simplicity, let $\a_i$ (with $1\leq i \leq t$) be the roots of $F(x)$ so that $\a_i$ (with $1\leq i \leq s$) are the roots of $F_1(x)$ (with $\a=\a_1$) and $\a_j$ (with $s<j\leq t$) are the roots of $G(x)$. Then $G(\alpha)=G(\alpha_1)=\prod_{j=s+1}^t (\alpha_1-\alpha_j)$ and $D=\prod_{1\leq i<j\leq t}(\a_i-\a_j)^2$.} $w(G(\alpha))=0$ and hence $(G(\alpha))=\cO_1$.
Now, we consider the natural evaluation map given by $\alpha$:
\begin{equation*}
\textnormal{ev}_\alpha: \cO_{K_v}[x] \surj \cO_1=\cO_{K_v}[\alpha] \simeq \cO_{K_v}[x]/{(F_1(x))},
\end{equation*}
and the induced isomorphism:
\begin{equation*}
\cO_{K_v}[x]/{(F_1(x), G(x))} \simeq \cO_1/{(G(\alpha))}=1.
\end{equation*}
Thus, $F_1(x)$ and $G(x)$ are relatively prime and therefore we have an isomorphism:
\begin{equation*}
\cO_{K_v}[x]/{(F(x))} \simeq \cO_{K_v}[x]/{(F_1(x))} \times \cO_{K_v}[x]/{(G(x))}.
\end{equation*}
Since the discriminant of $F_i(x)$ is a unit in $\cO_{K_v}$ for any $1\leq i \leq n-1$, we can apply the same argument successively. Accordingly, we get
\begin{equation*}
\cO_{K_v}[x]/{(F(x))}\simeq \prod_{i=1}^n \cO_{K_v}[x]/{(F_i(x))}.
\end{equation*}
This completes the proof.
\end{proof}

\begin{remark}
By Theorem \ref{theorem: odd niceness} and Proposition \ref{proposition: valuation is at most 1}, one can see that the elliptic curves studied by Li \cite{Li19} (see Assumption 2.1 there) are nice. 
\end{remark}

\ms
From now on, we study a generalization of the work of Brumer and Kramer \cite{BK77} to the case without the assumption $K_v/{\Q_2}$ is unramified. In other words, we discuss criteria when $E$ has semistable reduction at $v$.
\subsection{Good reduction}
Our main theorem in this subsection is the following.

\begin{theorem}\label{theorem: good reduction}
Suppose that $E$ has good reduction at $v$.
\begin{enumerate} 
\item
If $E$ has ordinary reduction at $v$, then $E$ is nice at $v$. 

\item
Suppose that $E$ has supersingular reduction at $v$ and $e$ is not divisible by $3$. If $v(a_1)$ is odd or $3v(a_1) \geq 2e$, then $L_v$ is a cubic ramified extension of $K_v$ and hence $E$ is nice at $v$. 
\end{enumerate}
\end{theorem}
\begin{proof}
First, suppose that $E$ has ordinary reduction at $v$. By Lemma \ref{lemma: criterion for ord vs ss} below, we have $v(a_1)=0$. By change of variables
$x\mapsto a_1^2 x -a_1^{-1}a_3$ and $y\mapsto a_1^3 y$, we have a new minimal model of the form
\begin{equation*}
y^2+xy=x^3+a_2' x^2+a_4' x +a_6'.
\end{equation*}
Then the $x$-coordinates of points of order two satisfy
\begin{equation*}
F(x)=x^3+(1/4+a_2')x^2+ a_4'x+a_6'=0.
\end{equation*}
Let $\a, \b$ and $\g$ be three roots of $F$. By Hensel's lemma, we may take
\begin{equation*}
\a=-1/4-a_2'+4a_4'+O(16) \in K_v
\end{equation*}
and $\b, \g \in O(2)$, where $t=O(s)$ means $v(ts^{-1}) \geq 0$.\footnote{There is a sign typo in the expression of $\a$ in proof of Lemma 3.5 of \cite{BK77}.} 

We claim that $E$ is upper nice at $v$. In other words, for any $P \in E(K_v)$ the valuations of $x(P)-\a$, $x(P)-\b$ and $x(P)-\g$ are all even. Let $P\in E(K_v)$. Then there is a point $Q \in \wt{E}(\ov{k})$ such that $2Q=\wt{P}$. In fact, we can take a finite extension $k'$ of $k$ so that $Q \in \wt{E}(k')$. Let $K'$ be the unramified extension of $K_v$ whose residue field is $k'$.
By the commutative diagram with exact rows
\begin{equation*}
\xyv {1.5}
\xymatrix{
E_1(K_v)/{2E_1(K_v)} \ar[r] \ar[d] & E(K_v)/{2E(K_v)} \ar[r] \ar[d]^-g & \wt{E}(k)/{2\wt{E}(k)} \ar[r] \ar[d] & 0\\
E_1(K')/{2E_1(K')} \ar[r]^-{f} & E(K')/{2E(K')} \ar[r] & \wt{E}(k')/{2\wt{E}(k')} \ar[r] & 0,
}
\end{equation*}
it is easy to see that $g([P]) \in \im (f)$. Consider another commutative diagram 
\begin{equation*}
\xyv{1.5}
\xymatrix{
E(K_v)/{2E(K_v)} \ar[r]^-{\d_{K_v}} \ar[d]^-g & \sel {L_v} \ar[d] \\
E(K')/{2E(K')} \ar[r]^-{\d_{K'}} & \sel {L'},
}
\end{equation*}
where $L'=K'[T]/{(F(x))}$. If $\d_{K'}(g([P])) \in \sel {\cO_{L'}}$ then $\d_{K_v}([P]) \in \sel {\cO_{L_v}}$ because $K'/{K_v}$ is unramified. Thus, to prove that $E$ is upper nice at $v$, it suffices to prove that 
for any $P \in E_1(K_v)$, the valuations of $x(P)-\b$ and $x(P)-\g$ are both even.\footnote{
If so, the valuation of $x(P)-\a$ is also even because $y(P)^2=F(x(P))=(x(P)-\a)(x(P)-\b)(x(P)-\g)$.}
By \cite[Ch. VII, Prop. 2.2]{Si09}, for $P(z) \in E_1(K_v)$ we have
\begin{equation*}
\begin{split}
x(P(z))-\b&=z^{-2}(1-z-(a_2'+\b) z^2 + O(z^3)) \qa \\
x(P(z))-\g&=z^{-2}(1-z-(a_2'+\g) z^2 + O(z^3)).
\end{split}
\end{equation*}
Since $\b, \g \in O(2)$, the valuations of $x(P(z))-\b$ and $x(P(z))-\g$ are even. This proves the claim.

Next, by Lemmas \ref{lemma: ordinary split}, \ref{lemma: ordinary non-split} and \ref{lemma: ordinary non-split ramified} below $E$ is lower nice at $v$.

Lastly, suppose that $E$ has supersingular reduction at $v$ and $e$ is not divisible by $3$. 
We claim that $L_v$ is a cubic ramified extension of $K_v$ (and hence $E$ is nice at $v$ by Proposition \ref{proposition: cubic extension}) if either $v(a_1)$ is odd or $3v(a_1)\geq 2e$. Suppose that $L_v$ is not a cubic ramified extension of $K_v$. We will derive a contradiction under the assumption that either $v(a_1)$ is odd or $3v(a_1) \geq 2e$.
Let $\a, \b$ and $\g$ be the roots of 
\begin{equation*}
F(x)=x^3+(a_1^2/4+a_2)x^2+({a_1a_3}/2+a_4)x+(a_3^2/4+a_6)=0.
\end{equation*}
(Note that $y^2=F(x)$ is a model of the given elliptic curve.)
Since $L_v$ is not a cubic ramified extension of $K_v$, we may assume that $v(\a) \in \Z$ and $v(\b), v(\g) \in \frac{1}{2}\Z$.
Note that since $\wt{E}$ is supersingular, we have $v(a_1)>0$ and $v(a_3)=0$ by Lemma \ref{lemma: criterion for ord vs ss} below.
Suppose that $v(a_1) \geq v(2)$. 
Since
\begin{equation*}
F(\a)=\a^3+(a_1^2/4+a_2)\a^2+(a_1a_3/2+a_4)\a+(a_3^2/4+6)=0,
\end{equation*}
there are at least two terms which have the smallest valuation among others.  
By our assumption, we have $v(a_1^2/4+a_2) \geq 0$ and $v(a_1a_3/2+a_4) \geq 0$. Since $v(a_3^2/4+a_6)=-2e <0$, we have $3v(\a)=-2e$, which is a contradiction because $e$ is not divisible by $3$.

For simplicity, let $m=v(a_1)$ and $n=v(\a)$. Suppose that $0< m < e$. Then $v(a_1^2/4+a_2)=2(m-e) < v(a_1a_3/2+a_4)=m-e<0$. Since $F(\a)=0$,we have $n<0$ (otherwise $F(\a)$ would have valuation $-v(4)<0$). Also, since 
\begin{equation*}
2(n+m-e)=v((a_1^2/4+a_2)\a^2) < v((a_1a_3/2+a_4)\a)=n+m-e, 
\end{equation*}
we must have either $n=2(m-e)$ or $3n>-2e$ (and $m=-n$). Thus, if $3m \geq 2e$ then the latter cannot happen and hence $n=2(m-e)$. Similarly, we get $v(\b)=v(\g)=2(m-e)$. This is a contradiction because $v(\a\b\g)=6(m-e)\neq -2e$. 
Lastly, if $3m<2e$ then we have $\{v(\a), v(\b), v(\g)\} \subset \{2(m-e), -m \}$. Since $v(\a\b\g)=-2e$, 
we may arrange $\a, \b, \g$ so that $v(\a)=2(m-e)$  and $v(\b)=v(\g)=-m$. Since $v(a_1 \b+a_3) \geq 0$ and 
\begin{equation*}
\begin{split}
F(\b)&=\b^3+(a_1^2/4+a_2)\b^2+(a_1a_3/2+a_4)\b+(a_3^2/4+a_6)\\
&=(\frac{a_1\b+a_3}{2})^2+\b^3+a_2\b^2+a_4\b+a_6=0,
\end{split}
\end{equation*}
we have $2(v(a_1\b+a_3)-e)=3v(\b)=-3m$, which is a contradiction if $m$ is odd. This completes the proof.
\end{proof}

\begin{lemma}\label{lemma: criterion for ord vs ss}
Suppose that $E$ has good reduction at $v$. Then either $v(a_1)=0$ or $v(a_3)=0$. Furthermore, $E$ has supersingular reduction at $v$ if and only if $v(a_1)>0$.
\end{lemma}
\begin{proof}
Since $E$ has good reduction at $v$, $v(\D^{\text{min}})=0$ by \cite[Ch. VII, Prop.5.1(a)]{Si09}, where $\D^{\text{min}}$ is the discriminant of a minimal model (\ref{equation : minimal model}). Suppose that $v(a_1) >0$ and $v(a_3)>0$. Then by the formula on page 42 of \textit{op. cit.}, we have $v(b_2) > 0$ and $v(b_6)>0$. Thus, $v(\D^{\text{min}})>0$, which is a contradiction. So we have either $v(a_1)=0$ or $v(a_3)=0$.

Next, suppose that $E$ has supersingular reduction at $v$. Since there is a unique supersingular elliptic curve $E_{\text{ss}}: y^2+y=x^3$ over $\ov{\F_2}$ (cf. page 148 of \textit{op. cit.}), we have $E \times_{\cO_{K_v}} \ov{\F_2} \simeq E_{\text{ss}}$. Since the coordinate change given by 
\begin{equation*}
x=u^2 x'+r \qa y=u^3 y'+u^2 sx'+t \text{ with } u \in \cO_{K_v}^\times
\end{equation*}
makes $ua_1'=a_1+2s$ and $u^3 a_3'=a_3+ra_1+2t$, we have $v(a_1')=0$ if and only if $v(a_1)=0$. Since $a_1'=0$ for $E_{\text{ss}}$, we must have $v(a_1) >0$. (Similarly, we get $v(a_3)=0$.)

Lastly, suppose that $v(a_1)>0$. Then $v(b_2)>0$ and hence $v(c_4)>0$. Thus, the $j$-invariant of the reduction $\wt{E}$ is $0$ and hence it has good supersingular reduction (cf. Exercise 5.7 of Chapter V in \textit{op. cit.}) This completes the proof.
\end{proof}

Below we use the same notation as in Section \ref{subsection: finite primes}.
\begin{lemma}\label{lemma: ordinary split}
Suppose that $E$ has ordinary reduction at $v$ and $\phi_v : L_v \simeq K_v \times K_v \times K_v$. 
Then we have
\begin{equation*}
\im (\d_{K_v})=\{ ([1], [a], [a]), ([\boxtimes], [a], [a\boxtimes]) : a \in \cO_{K_v}^\times \}.
\end{equation*}
In particular, $E$ is lower nice at $v$.
\end{lemma}
\begin{proof}
We use the same notation as in the proof of Theorem \ref{theorem: good reduction}. 
Since $E$ is upper nice at $v$, by \cite[p. 717]{BK77} the image of $\d_{K_v}$ is contained in
\begin{equation*}
\{ ([a], [b], [ab]) : a, b \in \cO_{K_v}^\times \}.
\end{equation*}
By Lemma \ref{lemma: selmer size at 2} we have $|\im (\d_{K_v})|=2^{d+2}$.
Since $|\sel {\cO_{K_v}}|=2^{d+1}$ and $(1+4\cO_{K_v})/{(\cO_{K_v}^\times)^2}=\{ [1], [\boxtimes] \}$, by counting argument it suffices to show that the first component of $\d_{K_v}([P])$ for any $P \in E(K_v)$ is contained in $1+4\cO_{K_v}$ modulo squares.
Consider the exact sequence
\begin{equation*}
\xymatrix{
E_1(K_v)/{2E_1(K_v)} \ar[r] & E(K_v)/{2E(K_v)} \ar[r]  & \wt{E}(k)/{2\wt{E}(k)} \ar[r]  & 0.\\
}
\end{equation*}
Since $|E(K_v)[2]|=4$ and $|E_1(K_v)[2]|=2$, we have $|\wt{E}(k)[2]|=2$. Since $\wt{E}(k)$ is finite, $|\wt{E}(k)/{2\wt{E}(k)}|=2$ and hence $E(K_v)/{2E(K_v)}$ is generated by $E_1(K_v)/{2E_1(K_v)}$ and $[Q]$ for some $Q \in E(K_v)$ with $\wt{Q} \not\in 2\wt{E}(k)$. 

First, since $\wt{Q} \neq \wt{\cO}$ the $x$-coordinate $x(Q)$ belongs to $\cO_{K_v}$. Thus, we have 
\begin{equation*}
x(Q)-\a \equiv 1/4(1+4a'_2+4x(Q)) \equiv 1+4u \text{ (modulo squares)}.
\end{equation*}

Next, let $P\in E_1(K_v)$. As on \cite[p. 720]{BK77} the second and third components of $\d_{K_v}(P)$ are 
\begin{equation*}
x(P)-\b \equiv s-\b z^2 ~\text{(modulo squares)} \qa x(P)-\g \equiv s-\g z^2 \text{ (modulo squares)}
\end{equation*}
for some $s=1-z+O(z^2) \in \cO_{K_v}^\times$ and $z\in (\pi)$. 
Since $\b+\g=-(a_2'+1/4)-\a=-4a_4'+O(16) \in O(4)$ and $\b\g \in O(4)$, we have
\begin{equation*}
(x(P)-\b)(x(P)-\g) \equiv s^2-(\b+\g)z^2s+\b\g z^4 \equiv  s^2 \equiv 1 \text{ (modulo squares)}.
\end{equation*}
Thus, the first component of $\d_{K_v}([P])$ is $[1]$. This proves the first assertion.

Lastly, by taking $a=1$ or $a=\boxtimes \in 1+4\cO_{K_v}$ we get $M_{1, v} \subset \im (\d_{K_v})$.
Thus, $E$ is lower nice at $v$. 
\end{proof}


For an extension $L_w/K_v$, recall the map $\text{Nm}$ defined in Lemma \ref{lemma : norm relation at 2}

\begin{lemma}\label{lemma: ordinary non-split}
Suppose that $E$ has ordinary reduction at $v$ and $L_w=K_v(\sqrt{\D})$ is an unramified quadratic extension of $K_v$ so that $\phi_v : L_v \simeq K_v \times L_w$. Then we have
\begin{equation*}
\im (\d_{K_v}) = \{ ([1], [a]), ([\boxtimes], [a \boxtimes']) : [a] \in \ker (\text{Nm}) \}.
\end{equation*}
In particular, $E$ is lower nice at $v$.
\end{lemma}
\begin{proof}
As in Lemma \ref{lemma: ordinary split}, if $P \in E_1(K_v)$ then 
the norm of the second component of $\d_{K_v}([P])$ must be a square. Thus,  the first component of $\d_{K_v}([P])$ is $[1]$. 
Also, if $Q \in E(K_v) \sm E_1(K_v)$ then the first component of $\d_{K_v}([Q])$
is of the form $1+4u$ with $u\in \cO_{K_v}$. Thus, we have
\begin{equation*}
\im (\d_{K_v}) \subset \{ ([1], [a]), ([\boxtimes], [ax] ) : [a] \in \ker(\text{Nm}) \}
\end{equation*}
for some $x \in \cO_{L_w}^\times$ such that $\text{Nm}([x])=[\boxtimes]$. 
By Lemma \ref{lemma : norm relation at 2}, $\text{Nm}([\boxtimes'])=[\boxtimes]$ and hence we can take $x = \boxtimes'$. Since $L_w$ is unramified, $|\ker (\text{Nm})|=2^d$. Thus, by counting argument we have the equality, which proves the first assertion. By taking $a=1$, we prove that $E$ is lower nice at $v$.
\end{proof}

\begin{lemma}\label{lemma: ordinary non-split ramified}
Suppose that $E$ has ordinary reduction at $v$ and $L_w=K_v(\sqrt{\D})$ is a ramified quadratic extension of $K_v$ so that $L_v \simeq K_v \times L_w$. If $ ~[\boxtimes] \not\in \text{im(Nm)}$ then we have
\begin{equation*}
\im (\d_{K_v})=\{ ([1], [a]) : [a] \in \ker(\text{Nm}) \}.
\end{equation*}
Otherwise, we have
\begin{equation*}
\im (\d_{K_v}) \subset \{ ([1], [a]), ([\boxtimes], [ax]) : [a] \in \ker(\text{Nm}) \},
\end{equation*}
where $x$ is taken so that $\text{Nm}([x])=[\boxtimes]$. 
In both cases, $E$ is lower nice at $v$.
\end{lemma}
\begin{proof}
As in Lemma \ref{lemma: ordinary non-split ramified}, we have
\begin{equation*}
\im (\d_{K_v}) \subset \{ ([1], [a]), ([\boxtimes], [ax]) : [a] \in \ker(\text{Nm}) \},
\end{equation*}
for some $x \in \cO_{L_w}^\times$ such that $\text{Nm}([x])=[\boxtimes]$. Thus, if $[\boxtimes] \not\in \text{im(Nm)}$ then such $x$ does not exist. Since $L_w/{K_v}$ is ramified, we have $|\ker(\text{Nm})|=2^{d+1}$ and hence $\im (\d_{K_v})=\{ ([1], [a]) : [a] \in \ker(\text{Nm}) \}$, as claimed.

To prove that $E$ is lower nice at $v$, it suffices to find a point $P \in E(K_v)$ such that
$\d_{K_v}([P])=([1], [\boxtimes'])$. 
Indeed, we can take $z=-4u$ for some $u \in \cO_{K_v}$ such that $1+4u$ is not a square, and $P=P(z)\in E_1(K_v)$. Then we have
\begin{equation*}
x(P(z))-\b=z^{-2}(1-z-(a_2'+\b)z^2+O(z^3)) \equiv 1+4u \equiv \boxtimes' \text{ (modulo squares)}.
\end{equation*}
Thus, we have $\d_{K_v}([P(z)])=([1], [\boxtimes'])$, as desired.
\end{proof}

\ms
\subsection{Multiplicative reduction}
In this subsection, we consider the case of multiplicative reduction. 
\begin{theorem}\label{theorem : multiplicative reduction}
Suppose that $E$ has multiplicative reduction at $v$. If $v(D)$ is odd, then $E$ is nice at $v$.
\end{theorem}
\begin{proof}
To prove the theorem, we need a description of the image of $\d_{K_v}$. 
By our assumption, $L_w=K_v(\sqrt{D})$ is a ramified quadratic extension and so we use the same notation as in Lemma \ref{lemma: ordinary non-split ramified}.
We claim that
\begin{equation*}
\im (\d_{K_v})=\{ ([1], [a]) : [a] \in \ker(\text{Nm}) \}.
\end{equation*}
Indeed, let $\cS:=\im(\sel {\cO_{K_v}} \inj \sel {\cO_{L_w}} )$. Then by Propositions 4.1 and the proof for Case 1 of Proposition 4.3 in \cite{BK77}, we can deduce
\begin{equation*}
\im (\d_{K_v})=\begin{cases}
\{([1], [z]) : [z] \in \cS \} & \text{ if $E$ has split multiplicative reduction at $v$},\\
\{ ([1], [z]) : [z] \in \ker(\text{Nm}) \} & \text{ otherwise}.
\end{cases}
\end{equation*}
Thus, it suffices to show that $\cS=\ker(\text{Nm})$ as subgroups of $\sel {L_w}$. Let $[\a] \in \cS$. Since $L_w/{K_v}$ is quadratic, we have $[\a] \in \ker(\text{Nm})$, i.e., $\cS \subset \ker(\text{Nm})$. Since $L_w$ is a ramified quadratic extension of $K_v$, we have $|\ker (\text{Nm})|=2^{d+1}$, which is equal to $|\cS|$. Therefore $\cS=\ker(\text{Nm})$ and hence the claim follows.

By the description of the image of $\d_{K_v}$ and Case 2 in Section \ref{section: local conditions}, it is easy to see that $E$ is nice at $v$, as desired.
\end{proof}

\begin{remark}
By \cite[Prop. 2(a) and Prop. 7]{Kr81}, the local condition $\im (\d_{K_v})$ does not change if we twist $E$ by an unramified quadratic extension under our assumption $v(D)$ is odd and $E$ has multiplicative reduction.
More generally, the same is true under the assumption that the local Tamagawa number is odd by the proof of case (3) of \cite[Lem. 5.9]{KL19}.
\end{remark}

\ms
\section{Examples}\label{section: examples}
Throughout this section, we choose a real quadratic field $K$ so that
\begin{equation*}
C_K = \{ 1\} \qa C_K^+ \simeq \zmod 2.
\end{equation*}
Also, we take $K$ so that it is ramified (resp. unramified) at $2$ in the case of good (resp. multiplicative) reduction. Furthermore, we take $F(x) \in \Q[x]$ so that the discriminant of $F$ is negative. Then we have $C_L^\infty=C_L^+$.
In the tables below, we use the following notation.
\begin{itemize}[--]
\item
$\D$ is the minimal discriminant of $E/\Q$.
\item
$m$ is the number of prime divisors of $\D$ inert in $K$.
\item
$n=\dim_{\F_2} C_L^\infty[2] - \dim_{\F_2} C_K^+[2]=\dim_{\F_2} C_L^+[2]-1$.
\item
$[n_1, \dots, n_r]$ is the group isomorphic to $\zmod {n_1} \times \cdots \times \zmod {n_r}$.

\item
$r_1$ (resp. $r_2$) is the rank of $E(\Q)$ (resp. $E^K(\Q)$), where $E^K$ is the quadratic twist of $E$ by $K$. It is often undetermined by the $2$-Selmer rank of $E/\Q$. In that case, we write its possible values in the table. 
\item
$s(E)$ is the $2$-Selmer rank of $E/K$, i.e., $s(E)=\dim_{\F_2} \Sel_2(E/K)$. 
\item
We say it is of \textit{type} $P$ (resp. $R$) if $s(E) \not\equiv n \pmod 2$ (resp. if $s(E) \equiv n \pmod 2$ and $r_1+r_2>n$).
\end{itemize}

In SAGE \cite{Sa20}, Simon's two descent code is used for computing the $2$-Selmer rank and the Mordell--Weil rank. Note that our computation of the $2$-Selmer rank is indirect because SAGE cannot compute most of $s(E)$ in the table (For instance, the computation of the $2$-Selmer rank for the case $a_6=37$ (good ordinary) already took more than a week. In general, the computation becomes more difficult if $a_6$ is getting large.)
Instead, we verify our computation as follows. Since $n \leq s(E) \leq n+2$, if it is of type $P$, in which case the $2$-Selmer rank is determined by the parity, we have $s(E)=n+1$. Also, since $s(E) \geq \text{rank of }E(K)$, which is $r_1+r_2$, if it is of type $R$, in which case the $2$-Selmer rank is determined by the rank, then we have $s(E)=n+2$. 

Although Simon's two descent code for elliptic curves over $\Q$ is very fast, that for elliptic curves over $K$ is very slow. Thus, our theorem tells a way to enhance the algorithm for general number fields under suitable assumptions on $E$ because $M_2$ is much smaller than $L(S, 2)$.

\subsection{Good reduction}
Let $K=\Q(\sqrt{3})$. First, we start with an elliptic curve $E/K$ given in the form
\begin{equation*}
y^2+a_1xy+a_3y=x^3+a_2x^2+a_4x+a_6 \qa a_i \in \Z.
\end{equation*}
Suppose that $E$ has good reduction at even primes. Then we have the following.
\begin{lemma} 
\label{lemma : 5.1} 
Suppose that $\D$ is squarefree and not divisible by $3$. Then $m$ has the same parity as $\dim_{\F_2} \Sel_2(E/K)$.
\end{lemma}
\begin{proof}
Let $\ve(E/K)$ be the root number of $E/K$ and let $\text{rk}_2(E/K)$ be the $2^\infty$-Selmer rank of $E$, as in \cite{DD11}. 
Then by Corollary 1.6 of \textit{op. cit.} we have $(-1)^{\text{rk}_2(E/K)}= \ve(E/K)$. Note that
\begin{equation*}
\dim_{\F_2} \Sel_2(E/K) - \text{rk}_2(E/K) = \dim_{\F_2} ((\Sha/\Sha_\text{div})[2]),
\end{equation*}
where $\Sha$ is the Shafarevich--Tate group of $E/K$ and $\Sha_\text{div}$ is the divisible subgroup (conjecturally trivial) of $\Sha$. 
This is an even number by the Cassels--Tate pairing (cf. \cite[Ch. X, Th. 4.14]{Si09}). Thus, it suffices to show that $(-1)^m=\ve(E/K)$. 

Since $\D$ is squarefree, $E$ is semistable. Thus, we have $\ve(E/K)=(-1)^{s+t}$, where $s$ is the number of the infinite places of $K$ and $t$ is the number of the primes where $E$ has split multiplicative reduction (cf. \cite[Sec. 1.2]{DD11}). Thus, it suffices to prove that $m \equiv s+t \equiv t \pmod 2$.

Let $v$ be a prime divisor of $\D$, and let $p$ be the prime number lying below $v$. 
Suppose first that $p$ is split in $K$. Then there is another prime $v'$ of $K$ lying above $p$. Since $E$ is defined over $\Q$, if $E/{K_v}$ has split multiplicative reduction then the same is true for $E/{K_{v'}}$. Next, suppose that $p$ is inert in $K$. Again, since $E$ is defined over $\Q$ and $K_v$ is an unramified quadratic extension of $\Q_p$, $E/{K_v}$ has always split multiplicative reduction. Thus, we have $m\equiv t \pmod 2$, as desired.
\end{proof}

Now, we take $a_1=0$ and $a_3=1$. Then $E/K$ has supersingular reduction at any even prime $v$. For simplicity, we further take $a_2=a_6=0$. Then by change of coordinates we have
\begin{equation*}
y^2=x^3+16a_4 x+16=F(x).
\end{equation*}
By SAGE \cite{Sa20} we have the following (good supersingular reduction at even primes).

\footnotesize
\begin{equation*}
\begin{array}{|c|c|c|c|c|c|c|c|c|c|}
\hline
a_4 & \D &  m &  C_L^+ & n & r_1 & r_2  & s(E) & \text{Type}  \\ \hline 
1 & -7\cdot 13 & 1 & [2] & 0  & 1 & 0 & 1 & P\\ \hline
4 & -7\cdot 19 \cdot 31 & 3 & [4, 2] & 1  & 2& 1 & 3 & R\\ \hline
5 & -23\cdot 349 & 0 & [2, 2] & 1 & 2 & 0, 2 & 2 & P \\ \hline
7 & -31 \cdot 709 & 1 & [6] & 0  & 1 & 0 & 1 & P\\ \hline
13 & -5 \cdot 11\cdot 2557 & 1 &  [210, 2] & 1 & 2 &1  &3 & R \\ \hline
14 & -13 \cdot 59 \cdot 229& 0 &  [60, 2] & 1  & 1& 1   & 2 & P\\ \hline
17 & -43 \cdot 71 \cdot 103 & 2 & [20, 2, 2, 2] & 3 & 3 & 1, 3 & 4 & P \\ \hline
19 & -79 \cdot 5557 & 1 &   [16, 4, 2] & 2  & 3 & 0, 2  & 3 & P\\ \hline 
22 & -7 \cdot 13 \cdot 7489 & 1  & [28, 2, 2] & 2   & 2 & 1 & 3 & P \\ \hline
\end{array}
\end{equation*}
\begin{equation*}
\begin{array}{|c|c|c|c|c|c|c|c|c|c|}
\hline
a_4 & \D &  m &  C_L^+ & n & r_1 & r_2  & s(E) & \text{Type}  \\ \hline 
23 & -5 \cdot 7 \cdot 19 \cdot 1171 & 4 & [24, 2] & 1 & 2 & 0, 2 & 2 & P\\ \hline
25 & -7 \cdot 19\cdot 73 \cdot 103 & 3 & [78] & 0 &  1 & 0 & 1 & P\\ \hline
26 & -107 \cdot 10513 & 0 & [10, 2, 2] & 2 & 2 & 2 & 4 & R \\ \hline
31 & -127 \cdot 15013 & 1 & [42] & 0 & 1 & 0 & 1 & P \\ \hline
32 & -7 \cdot 131 \cdot 2287 & 2 & [24, 2, 2] & 2 & 3 & 1, 3 & 4 & R \\ \hline
34 & -139 \cdot 18097 & 1 & [60, 2, 2, 2] & 3 & 3 & 2& 5 & R \\ \hline 
35 & -11 \cdot 13 \cdot 31 \cdot 619 & 2 & [78, 6] & 1 & 2 & 0, 2 & 2 & P \\ \hline
37 & -7 \cdot 151 \cdot 3067 & 3 & [52, 2, 2] & 2 & 2 & 1 & 3 & P\\ \hline
40 & -13 \cdot 163 \cdot 1933 & 1 & [40, 2] & 1 & 2 & 1 & 3 & R \\ \hline
41 & -61 \cdot 167 \cdot 433 & 0 & [10, 2, 2, 2, 2] & 4 & 3 & 3 & 6 & R \\ \hline
44 & -7 \cdot 19 \cdot 179 \cdot 229 & 2 & [44, 2, 2, 2] & 3 & 2 & 2 & 4 & P \\ \hline
\end{array}
\end{equation*}
\normalsize

\begin{remark}
When $a_4=5, 17, 19, 23, 32$ and $35$, we deduce that $r_2=0, 1, 0, 0, 1$ and $0$, respectively. 
\end{remark}

\ms
Next, we take $a_1=1$ and $a_3=a_2=a_4=0$. By direct computation, the discriminant of $E$ is $-a_6(1+432a_6)$. Thus, it has ordinary reduction at any even prime $v$ if $v(a_6) \in 12\Z$. By change of coordinates we have
\begin{equation*}
y^2=x^3+ x^2+64a_6=F(x).
\end{equation*}
By SAGE \cite{Sa20} we have the following (good ordinary reduction at even primes).

\footnotesize
\begin{equation*}
\begin{array}{|c|c|c|c|c|c|c|c|c|}
\hline
a_6 & \D &  m & C_L^+ & n & r_1 & r_2 & s(E) & \text{Type}\\ \hline
1& - 433& 0 &  [4, 2]  & 1  & 2 & 0  & 2 & P\\ \hline
5& -  5 \cdot 2161& 1  &  [14] & 0 & 1& 0 & 1 & P\\ \hline
13&-13 \cdot 41 \cdot 137 & 2 &  [2, 2] & 1  & 2 & 0  & 2 & P \\ \hline
19&-19 \cdot  8209  & 1 &   [2, 2] & 1 & 1 & 2  & 3 & R \\ \hline
29&-11\cdot 17 \cdot 29\cdot 67& 3 & [2] &0 & 1 & 0 & 1 & P\\ \hline
37&-5 \cdot  23 \cdot  37 \cdot  139 &2 & [2] & 0 & 1 & 1   & 2 & R \\ \hline
41&-41 \cdot  17713 &1 &  [370] & 0& 1 & 0   & 1 & P \\ \hline
43&-13 \cdot  43 \cdot 1429 &1 &  [12, 2, 2]  & 2 & 2& 1 & 3 & P\\ \hline
47&-5 \cdot  31 \cdot  47 \cdot  131  & 2 &   [2, 2] &1 & 1 & 1 & 2 & P \\ \hline
53 & -7 \cdot 53 \cdot 3271 & 3 & [16, 2, 2, 2] & 3 & 2 & 3 & 5 & R \\ \hline
55 & -5 \cdot 11 \cdot 23761 & 1 & [4, 2, 2, 2] & 3 & 2 & 3 & 5 & R \\ \hline
65 & - 5 \cdot 13 \cdot 28081 & 1 & [6, 2] & 1 & 2 & 1 & 3 & R \\ \hline
73 & -11 \cdot 47 \cdot 61 \cdot 73 & 0 & [2, 2] & 1 & 1 & 1 & 2 & P\\ \hline
77 & - 5 \cdot 7 \cdot 11 \cdot 6653 & 3 & [2] & 0 & 1 & 0 & 1 & P\\ \hline
79 & -79 \cdot 34129 & 1 & [6, 2, 2] & 2 & 1 & 2 & 3 & P\\ \hline
89 & -89 \cdot 38449 & 1 & [2] & 0 & 1 & 0 & 1 & P\\ \hline
95 & -5 \cdot 7 \cdot 11 \cdot 13 \cdot 19 \cdot 41 & 4 & [2, 2] & 1 &1 & 1 & 2 & P\\ \hline 
101 & -101 \cdot 43633 & 1 & [22, 2] & 1 & 1 & 2 & 3 & R\\ \hline
103 & -103 \cdot 44497 & 1 & [1008, 2, 2] & 2 & 1 & 0 & 3 & P \\ \hline
113 &-113 \cdot 48817 & 1 & [26] & 0 & 1 & 0 & 1 & P \\ \hline
\end{array}
\end{equation*}
\normalsize

\subsection{Multiplicative reduction}
Let $K=\Q(\sqrt{21})$. As above, we take an elliptic curve $E/K$ given in the form
\begin{equation*}
y^2+a_1xy+a_3y=x^3+a_2x^2+a_4 x+a_6 \qa a_i \in \Z.
\end{equation*}
We take $a_1=1$ and $a_2=a_3=a_4=0$, so $F(x)=x^3+x^2+64a_6$. By Tate's algorithm \cite[p. 366]{Si94}, it is easy to see that $E$ has multiplicative reduction at any prime $v$ dividing $(a_6, \D)$. Thus, we take $a_6=2b$. We choose $b$ so that $1+864b$ is squarefree, which guarantees that $E/{K_v}$ has semistable reduction at any prime $v$. We also take $b$ so that $\D$ has odd or zero valuation at all primes, and $\D$ is even and prime to $21$.
Similarly as in Lemma \ref{lemma : 5.1}, we can easily deduce that $m$ has the same parity as $\dim_{\F_2} \Sel_2(E/K)$.

\footnotesize
\begin{equation*}
\begin{array}{|c|c|c|c|c|c|c|c|c|}
\hline
b & \D &  m & C_L^+ & n&  r_1 & r_2 & s(E) & \text{Type} \\ \hline 
1 & -2 \cdot 5 \cdot 173 & 1 & [28, 2] & 1 & 0 & 1 & 1, 3 &  \\ \hline
4 & -2^3 \cdot 3457 & 2 & [6] & 0 & 1 & 1 & 2 & R \\ \hline
5 & -2 \cdot  5 \cdot 29 \cdot 149 & 3  & [6] & 0 & 1 & 0 & 1 & P \\ \hline
11 & -2 \cdot 5 \cdot 11 \cdot 1901 & 3  & [2] & 0 & 1 & 0 & 1 & P \\ \hline
13 & -2 \cdot 13 \cdot 47 \cdot 239 & 3 & [2] & 0 & 1 & 0 & 1 & P \\ \hline

17 & -2 \cdot  17 \cdot 37 \cdot 397 & 2 & [2, 2] & 1 & 1 & 1 & 2 & P \\ \hline
19 & -2 \cdot 19 \cdot 16417 & 2 & [2] & 0 & 0 & 0 & 0, 2 & \\ \hline
20 & -2^3 \cdot 5 \cdot 11 \cdot 1571 & 2 & [26, 2, 2] & 2 & 1 & 1 & 2, 4 & \\ \hline
29 & -2 \cdot 29 \cdot  25057 & 2 & [2, 2] & 1 & 0 & 0 & 2 & P\\ \hline
31 & -2 \cdot 5\cdot 11\cdot 31 \cdot 487 & 3 & [6, 2] & 1 & 0 & 1 & 1, 3 & \\ \hline

43 & -2 \cdot 43 \cdot 53 \cdot 701 & 3 & [12, 2] & 1 & 1 & 0 & 1, 3 & \\ \hline
47 & -2 \cdot 47 \cdot 40609 & 1 & [2] & 0 & 0 & 1 & 1& P \\ \hline
52 & -2^3 \cdot 13 \cdot 179 \cdot 251& 3 & [2] & 0 & 1 & 0 & 1 & P\\ \hline
53 & -2 \cdot 11 \cdot  23 \cdot 53 \cdot 181 & 5 & [2, 2] & 1 & 0 & 1 & 1, 3 & \\ \hline
55 & -2 \cdot 5 \cdot 11 \cdot 47521 & 3 & [30, 2] & 1 & 1 & 2 & 3& R \\ \hline

59 & -2 \cdot 19 \cdot 59 \cdot 2683 & 2 & [2, 2] & 1 & 1 & 1 & 2 & P \\ \hline
61 & -2 \cdot 5 \cdot 61 \cdot 83 \cdot 127 & 2 & [ 2, 2] & 1 & 0, 2 & 0, 2 & 2 & P \\ \hline 
67 & -2 \cdot 13 \cdot 61 \cdot 67 \cdot 73& 4 & [20, 2, 2, 2] & 3 & 0, 2 & 0, 2 & 4 & P \\ \hline
68 & -2^3 \cdot 17 \cdot 41 \cdot 1433 & 1 & [2]& 0 & 1 & 0 & 1 & P \\ \hline
71 &  -2\cdot 5\cdot 71 \cdot 12269 & 2 & [2, 2] & 1 & 1& 1& 2 & P \\ \hline

73 & -2 \cdot 73 \cdot 63073 & 3 & [2, 2] & 1 & 0 & 1 & 1, 3 & \\ \hline
76 & -2^3 \cdot 5 \cdot 19 \cdot 23 \cdot 571 & 3 & [2] & 0 & 0 & 1 & 1 & P \\ \hline
83 &  -2 \cdot 83 \cdot 71713 & 2 & [2, 2, 2] & 2 & 0 & 2 & 2, 4 & \\ \hline
89 & -2 \cdot 89 \cdot 131 \cdot 587  & 1 & [2, 2] & 1 & 1 & 0 & 1, 3 & \\ \hline
92 & -2^3 \cdot 23 \cdot 29 \cdot 2741&4 & [2, 2] & 1 & 1 & 1& 2 & P \\ \hline

95 & -2 \cdot 5\cdot 19 \cdot 79 \cdot 1039 & 3 & [42, 2] & 1 & 0 & 1 & 1, 3 & \\ \hline
97 & -2 \cdot 11 \cdot 19 \cdot 97 \cdot 401 & 5  & [2] & 0 & 0 & 1 & 1 & P \\ \hline
101  & -2 \cdot 5 \cdot 31 \cdot 101 \cdot 563 & 2 & [2, 2, 2] & 2 & 0, 2 & 2 & 2, 4 & \\ \hline
103  & -2 \cdot 103 \cdot 88993 & 2 & [4, 2, 2] & 2 & 0, 2 & 2 & 2, 4 & \\ \hline
109 & -2 \cdot 41 \cdot 109 \cdot 2297 & 2 & [2, 2] & 1 & 1 & 1 & 2 & P \\ \hline

113 & -2 \cdot 89 \cdot 113 \cdot 1097 & 2 & [2, 2] & 1 & 1 & 1 & 2 & P\\ \hline
115 & -2 \cdot 5 \cdot 23 \cdot 67 \cdot 1483 & 3 & [2, 2] & 1 & 0 & 1 & 1, 3 & \\ \hline
124 & -2^3 \cdot 31 \cdot 107137 & 2 & [6] & 0 & 0 & 0 & 0, 2 &  \\ \hline
125 & -2 \cdot 5^3 \cdot 17 \cdot 6353 & 2 & [2, 2] & 1 & 1 & 1 & 2 & P \\ \hline
127 & -2 \cdot 127 \cdot 197 \cdot 557 & 3 & [10] & 0 & 1 & 0 & 1 & P \\ \hline

131 & -2\cdot 5 \cdot 131 \cdot 22637 & 1 & [2] & 0 & 1 & 0 & 1  & P \\ \hline
137 & -2 \cdot 137 \cdot 118369 & 3 & [18]  & 0 & 0 & 1 & 1 & P \\ \hline
139 & -2 \cdot 139 \cdot 120097 & 3 & [2, 2] & 1 & 0 & 1 & 1, 3 & \\ \hline
143 & -2\cdot 11 \cdot 13 \cdot 123553 & 4 & [2, 2] & 1 & 1 & 1 & 2 & P \\ \hline
145 & -2 \cdot 5 \cdot 13 \cdot 23 \cdot 29 \cdot 419 & 4 & [10] & 0 & 1 & 1 & 2 & R \\ \hline

148 & -2^3 \cdot 37 \cdot 127873 & 1 & [4, 2] & 1 & 2 & 1 & 3 & R \\ \hline
\end{array}
\end{equation*}
\normalsize

By SAGE \cite{Sa20} we can find all $b$ satisfying the conditions above in the range $1\leq b \leq 150$, which are exactly those in the first column of the table above. The number of elements of type $P$ is $22$, the number of elements of type $R$ is $4$, and the number of elements where our method cannot determine the $2$-Selmer rank is $15$.

\begin{remark}
When $b=67$, we deduce that $r_1=r_2=2$. On the other hand, when $b=61$ we cannot determine the exact value of the rank of $E(\Q)$. 
\end{remark}


\subsection*{Acknowledgments}
We thank Chao Li and Ariel Pacetti for answering questions on their papers \cite{Li19} and \cite{BPT}, respectively, and Denis Simon for helpful correspondences for his code. We also thank the referee for helpful comments and suggestions.

H.Y. was supported by Research Resettlement Fund for the new faculty of Seoul National University. He was also supported by the National Research Foundation of Korea (NRF) grant funded by the Korea government (MSIT) (No. 2019R1C1C1007169 and No. 2020R1A5A1016126). 
M.Y. was supported by a KIAS Individual Grant (SP075201) via the Center for Mathematical Challenges at Korea Institute for Advanced Study. He was also supported by the National Research Foundation of Korea (NRF) grant funded by the Korea government (MSIT) (No. 2020R1C1C1A01007604).

\end{document}